\documentclass[10pt]{amsart}
\usepackage{epsfig}
\usepackage{graphicx}
\usepackage{amscd}
\usepackage{amsmath}
\usepackage{amsxtra}
\usepackage{amsfonts}
\usepackage{amssymb}

\newtheorem{theorem}{Theorem}[section]
\newtheorem{corollary}[theorem]{Corollary}
\newtheorem{lemma}[theorem]{Lemma}
\newtheorem{proposition}[theorem]{Proposition}

\theoremstyle{definition}
\newtheorem{definition}[theorem]{Definition}
\newtheorem{remark}[theorem]{Remark}

\theoremstyle{remark}

\renewcommand{\theclaim}{\textup{\theclaim}}

\newtheorem*{acknowledgements}{Acknowledgements}

\numberwithin{equation}{section}

\newcommand\restr[2]{{% we make the whole thing an ordinary symbol
  \left.\kern-\nulldelimiterspace % automatically resize the bar with \right
  #1 % the function
  \vphantom{\big|} % pretend it's a little taller at normal size
  \right|_{#2} % this is the delimiter
  }}

\def\openone%{\hbox{\upshape \small1\kern-3.3pt\normalsize1}}

{\mathchoice

{\hbox{\upshape \small1\kern-3.3pt\normalsize1}}

{\hbox{\upshape \small1\kern-3.3pt\normalsize1}}

{\hbox{\upshape \tiny1\kern-2.3pt\SMALL1}}

{\hbox{\upshape \Tiny1\kern-2pt\tiny1}}}

\makeatletter

\newbox\ipbox

\newcommand{\ip}[2]{\left\langle #1\, , \,#2\right\rangle}
\newcommand{\diracb}[1]{\left\langle #1\mathrel{\mathchoice

{\setbox\ipbox=\hbox{$\displaystyle \left\langle\mathstrut
#1\right.$}

\vrule height\ht\ipbox width0.25pt depth\dp\ipbox}

{\setbox\ipbox=\hbox{$\textstyle \left\langle\mathstrut
#1\right.$}

\vrule height\ht\ipbox width0.25pt depth\dp\ipbox}

{\setbox\ipbox=\hbox{$\scriptstyle \left\langle\mathstrut
#1\right.$}

\vrule height\ht\ipbox width0.25pt depth\dp\ipbox}

{\setbox\ipbox=\hbox{$\scriptscriptstyle \left\langle\mathstrut
#1\right.$}

\vrule height\ht\ipbox width0.25pt depth\dp\ipbox}

}\right. }

\newcommand{\dirack}[1]{\left. \mathrel{\mathchoice

{\setbox\ipbox=\hbox{$\displaystyle \left.\mathstrut
#1\right\rangle$}

\vrule height\ht\ipbox width0.25pt depth\dp\ipbox}

{\setbox\ipbox=\hbox{$\textstyle \left.\mathstrut
#1\right\rangle$}

\vrule height\ht\ipbox width0.25pt depth\dp\ipbox}

{\setbox\ipbox=\hbox{$\scriptstyle \left.\mathstrut
#1\right\rangle$}

\vrule height\ht\ipbox width0.25pt depth\dp\ipbox}

{\setbox\ipbox=\hbox{$\scriptscriptstyle \left.\mathstrut
#1\right\rangle$}

\vrule height\ht\ipbox width0.25pt depth\dp\ipbox}

} #1\right\rangle}

\newcommand{\cj}[1]{\overline{#1}}

\newcommand{\bz}{\mathbb{Z}}
\newcommand{\M}{\mathcal{M}}

\newcommand{\br}{\mathbb{R}}
\newcommand{\bc}{\mathbb{C}}

\newcommand{\bn}{\mathbb{N}}

\newcommand{\Ds}{\mathsf{D}}
\newcommand{\nar}[1]{\stackrel{#1}{\rightarrow}}

\def\blfootnote{\xdef\@thefnmark{}\@footnotetext}

\renewcommand{\mod}{\operatorname{mod}}

\input xy
\xyoption{all}
\usepackage{amssymb}

\def\F{\mathcal{F}}

\def\H{\mathcal{H}}

\def\-{^{-1}}

\def\D{\mathcal{D}}

\def\U{\mathcal{U}}

\def\ty{\emptyset}

\newcommand{\gr}{\operatorname*{graph}}

\newcommand{\domain}{\operatorname*{domain}}

\begin{document}

\title[Unitary groups and spectral sets]{Unitary groups and spectral sets}
\author{Dorin Ervin Dutkay}
\blfootnote{}
\address{[Dorin Ervin Dutkay] University of Central Florida\\
    Department of Mathematics\\
    4000 Central Florida Blvd.\\
    P.O. Box 161364\\
    Orlando, FL 32816-1364\\
U.S.A.\\} \email{Dorin.Dutkay@ucf.edu}

\author{Palle E.T. Jorgensen}
\address{[Palle E.T. Jorgensen]University of Iowa\\
Department of Mathematics\\
14 MacLean Hall\\
Iowa City, IA 52242-1419\\}\email{jorgen@math.uiowa.edu}

\thanks{}

\begin{abstract}
We study spectral theory for bounded Borel subsets of $\br$ and in particular finite unions of intervals. For Hilbert space, we take $L^2$ of the union of the intervals. This yields a boundary value problem  arising from the minimal operator  $\Ds = \frac1{2\pi i}\frac{d}{dx}$  with domain consisting of $C^\infty$ functions vanishing at the endpoints. We offer a detailed interplay between geometric configurations of unions of intervals and a spectral theory for the corresponding self-adjoint extensions of $\Ds$ and for the associated unitary groups of local translations.
While motivated by scattering theory and quantum graphs, our present focus is on the Fuglede-spectral pair problem. Stated more generally, this problem asks for a determination of those bounded Borel sets $\Omega$ in $\br^k$  such that $L^2(\Omega)$ has an orthogonal basis of Fourier frequencies (spectrum), i.e., a total set of orthogonal complex exponentials restricted to $\Omega$.
 In the general case, we characterize Borel sets $\Omega$ having this spectral property in terms of a unitary representation of $(\br, +)$ acting by local translations.  The case of $k = 1$ is of special interest, hence the interval-configurations. We give a characterization of those geometric interval-configurations which allow Fourier spectra directly in terms of the self-adjoint extensions of the minimal operator $\Ds$. This allows for a direct and explicit interplay between geometry and spectra.
\end{abstract}
\maketitle \tableofcontents

\section{Introduction}

\newcommand{\Dmax}{\mathcal D_{\operatorname*{max}}}
\newcommand{\Dmin}{\Ds_{\operatorname*{min}}}

In this paper we study a classification problem (SAE) for self-adjoint extensions of Hermitian operators in Hilbert space, with dense domain and finite deficiency indices $(n, n)$. While this question has many ramifications, we will focus here on a restricted family of extension operators. We begin with a justification for the restricted focus.
\medskip

{\bf Definitions.} It turns out that this problem arises in a number of instances which on the face of it appear quite different, but turn out to be unitarily equivalent. While we have in mind models for scattering of waves on a disconnected obstacle, and quantum mechanical transition probabilities, it will be convenient for us to select the version of problem (SAE) where $\H = L^2(\Omega)$  and $\Omega$ is a bounded open subset of the real line  $\br$ with a finite number of components, i.e., $\Omega$ is a finite union of open disjoint intervals. We consider $\Ds := \frac{1}{2\pi i}\frac{d}{dx}$  corresponding to vanishing boundary conditions (the minimal operator). Then the deficiency indices are $(n, n)$ when $n$ is the number of components in $\Omega$. In a different context, mathematical physics, the minimal operator was considered in \cite{Jor81}.

Let $\H$ be a complex Hilbert space with inner product $\ip{\cdot}{\cdot}=\ip{\cdot}{\cdot}_{\H}$. Let $\mathcal D\subset\H$ be a dense subspace in $\H$. A linear operator $L$ defined on $\mathcal D$ is said to be {\it symmetric} (or {\it Hermitian}) iff 
$$\ip{Lf}{g}=\ip{f}{Lg}\mbox{ for all }f,g\in\mathcal D.$$
In this case, the adjoint operator $L^*$ is defined on a subspace $\domain (L^*)$ containing $\mathcal D$ and $L\subset L^*$, where "$\subset$" refers to containment of graphs.

If the dimensions of the two eigenspaces $\{f_\pm\in\domain(L^*) : L^*f_\pm=\pm if_\pm\}$ are equal (called {\it the deficiency indices }) then $L$ has self-adjoint extensions. Every self-adjoint extension $A$ of $L$ must satisfy $L\subset A\subset L^*$ and any such $A$ will be a restriction of $L^*$.

     In section 3, we offer a geometric model for the study of finite deficiency indices $(n, n)$. While this work is directly related to recent work \cite{JPTa, JPTb, JPTc}, our present focus is different, as are our themes. To make our present paper reasonably self-contained, it will be convenient for us to include here (section 3) some basic lemmas needed in the proof of our main theorems. When $n$ (the number of intervals in $\Omega$) is fixed, the set of all self-adjoint extensions of $\Ds$  is in bijective correspondence with the group $U_n$  of all complex unitary $n \times n$ matrices. Moreover we include in Proposition \ref{pr1.11} an explicit formula for our $U_n$ correspondence for the problem (SAE), expressed directly in terms of the  $2n$ interval-endpoints constituting the boundary of $\Omega$. It is an action by elements in the conformal group $U(n, n)$.

 {\bf Motivation.}     One motivation for our study is a spectral theoretic question (conjecture) raised first in a paper by Fuglede \cite{Fug74}. We refer to \cite{Fug74} and  \cite{JoPe96} for details, but, in summary, the question is whether the existence of a Fourier basis in $L^2(\Omega)$ is equivalent to the set $\Omega$ possibly tiling $\br^k$ by a set of translation vectors. Here the question addresses any dimension $k$, and for any Borel set $\Omega$ of finite positive Lebesgue measure. If there is a subset $\Lambda$ in $\br^k$ such that the complex $\Lambda$ exponentials form an orthogonal basis in $L^2(\Omega)$ we say that $(\Omega, \Lambda)$ is a spectral pair, and we say that the set $\Omega$ is spectral.

      There are a number of reasons for restricting to a one-dimensional model.

       In application to the Lax-Phillips model for scattering of acoustic waves on a bounded and disconnected obstacle in $\br^k$,  $k > 1$ (see e.g., \cite{PWW87}), the present case of one dimension will then represent a wave motion in a single direction in $\br^k$.  In the Lax-Phillips model for obstacle scattering, time-evolution of waves is represented by a unitary one-parameter group of operators acting on an associated energy Hilbert space $\H$.
       
       \medskip
       Below we outline briefly two reasons why our results about intervals are relevant to Lax-Phillips scattering theory \cite{LaPh89}.  Recall that Lax-Phillips scattering theory deals with scattering of acoustic waves around solid obstacles in $\br^k$, i.e., the solution to some wave equation, represented in a Hilbert space, and the study of wave solutions in the complement of a compact obstacle, e.g., for the wave equation in an exterior domain.
  This study of the exterior of a bounded obstacle in turn is divided into the case of a connected obstacle, vs the case of disconnected ones. The latter is more subtle because it yields intriguing configurations of bound-states, i.e., the trapping of waves in the bounded connected components in the complement obstacle: in mathematical language, eigenvectors and eigenvalues.
  Now working directly in $\br^k$  and in components in the complement of a bounded obstacle is typically difficult, both in the theory and in applications. One way around this difficulty goes via hyperplane segments in the obstacle, e.g., the use of a suitable Radon transform \cite{Hel11}. But rather, our present approach is to study instead waves traveling along varying linear directions in $\br^k$. These directions in turn are specified by lines passing through the obstacles.

  With disconnected obstacles in $\br^k$, the linear intersections will then be the union of a number of bounded intervals. Using a second fact from Lax-Phillips scattering theory  \cite{LaPh89}, we note that the solutions to the acoustic wave equations may be represented by a strongly continuous unitary one-parameter group, say $U(t)$ acting on an energy Hilbert space. But Lax and Phillips \cite{LaPh89} show that this group $U(t)$ may be taken to be unitarily equivalent to a translation representation. Hence, for each of the one-dimensional linear directions, we get an equivalent translation group generated by a self-adjoint operator arising in a one-dimensional boundary value problem; and hence we are led to self-adjoint extensions of a minimal derivative $\frac1{2\pi i}\frac d{dx}$ in some linear set $\Omega$ that is the union of a finite number of disjoint open intervals. Note that $\Omega$ will vary with the choice of different linear directions through some fixed disconnected obstacle.

  Starting with one fixed such set $\Omega$, we study the question of self-adjoint extensions. While our motivation derives from the problem of spectral pairs, the problem is of independent interest in scattering theory. And even in the study of possible spectral sets $\Omega$, one must consider the variety of all self-adjoint extensions, although only a few of these extensions, if any, yield spectral pairs. If some $\Omega$ is spectral, we show that as the spectrum part $\Lambda$ in a corresponding spectral pair, one may take for $\Lambda$ the spectrum of some self-adjoint extension. Only with hindsight one realizes that in fact ``most'' self-adjoint extensions are not spectral. Identifying those that are is a subtle problem.

\medskip 
 Our main results are in sections 2 through 4, and we give an application in section 5.
 Some highpoints: In Theorem \ref{th4.1} we show that every spectral pair $(\Omega, \Lambda)$, with $\Omega$ a bounded Borel set of positive Lebesgue measure, has a group of local translations. In Theorem \ref{pr4.8} we apply this to the case when $\Omega$ has a spectrum with period $p$, where $p$ is a fixed positive integer.
In Sections 3 and 4 we turn to the cases when $\Omega$ is assumed a union of a finite number of non-overlapping intervals. We introduce determinant-bundles, and we use them in Theorem \ref{th2.12} in order to classify the spectral types of self-adjoint extensions. This, and our local translation groups, are applied, in turn, in proving our results on spectral correspondence: In Theorem \ref{th3.9} and Corollary \ref{cor3.10} we offer a classification for the case when the spectrum has period $p$; and in Theorems \ref{th3.14} and \ref{th4.10} we further study the correspondence between two sets making up a spectral pair. 

We view the approach to spectral pairs via self-adjoint extensions as a tool for generating spectra. In fact, in the literature, so far there are rather few analytic and constructive tools available helping one produce spectral pairs. We present an analysis of self-adjoint extension as one such tool. Our purpose is to develop these self-adjoint extensions, refine them, and apply this to the spectral pair question. But these other applications to scattering theory are of general interest in mathematical physics.

  As for spectral pairs, we find that among all the self-adjoint extensions only a very small subset is spectral. And of course, for many cases of a linear set $\Omega$, this (spectral) subset may be empty.
But even if some $\Omega$ is not spectral, we have a detailed geometric configuration of self-adjoint extensions with associated spectra. We also study these, their nature and geometry.

       Now, Fuglede's conjecture is known to be negative when the dimension $k$ is 3 or more \cite{Tao04, KM06} , but the cases of $k = 1$ and $k = 2$  are still open. But in the plane ($k = 2$) the partial derivatives corresponding to zero boundary conditions for a fixed open planar $\Omega$  are known to have deficiency indices $(\infty,\infty)$. Moreover the geometric issues involved for the two cases $k = 1$ and $k = 2$ are quite different, and we thus focus here on $k = 1$.

We prove in section 3 the following theorem: Given a finite number  $n > 1$ of components in some fixed $\Omega$, and a pair, consisting of a matrix $B$ and $2n$ boundary points; i.e., given $B$ in $U_n$, and $2n$  interval endpoints (the boundary of $\Omega$), by passing to the corresponding self-adjoint extension $A_B$  in $L^2(\Omega)$ we get a spectral pair ($\Omega$, spectrum($A_B$))  if and only the matrix $B$ has a certain factorization in terms of two unitary matrices, one formed from the left-hand side interval endpoints, and the other from the right-hand side interval endpoints.

      There is a recent substantial prior literature on orthogonal Fourier exponentials and spectral duality, see for example \cite{Fug74, DuJo09a, DuJo09b, DuJo08a, DuJo07a, DuJo07b, JoPe98, JoPe96, JoPe99, LaSh92, IoPe98, IKT01, Tao04, KM06,MR2817339,JPTc,JPTd}.

\begin{definition}\label{def0.1}
For $\lambda\in\br$, we denote by $e_\lambda(t)=e^{2\pi i \lambda t}$, $t\in\br$. 
A Borel subset $\Omega$ of finite Lebesgue measure is said to be {\it spectral} if there is a set $\Lambda$ in $\br$ such that the family of exponential functions $\left\{\frac{1}{\sqrt{|\Omega|}}e_\lambda :\lambda\in\Lambda\right\}$ is an orthonormal basis for $L^2(\Omega)$. In this case, $\Lambda$ is called a {\it spectrum} for $\Omega$ and $(\Omega,\Lambda)$ is called a {\it spectral pair}.

$|\Omega|$ indicates the Lebesgue measure of $\Omega$.

A finite Borel measure $\mu$ on $\br$ is called {\it spectral} if there exists a set $\Lambda$ in $\br$ such that $\{e_\lambda : \lambda\in\Lambda\}$ is an orthogonal basis for $L^2(\mu)$. We call $\Lambda$ a {\it spectrum} for $\mu$. 

A finite set $A$ in $\br$ is spectral if the atomic measure $\frac{1}{|A|}\sum_{a\in A}\delta_a$ is spectral.

A Borel subset $\Omega$ of $\br$ {\it tiles} $\br$ by translations if there exists a subset $\mathcal T$ of $\br$ such that $(\Omega+t)_{t\in\mathcal T}$ forms a partition of $\br$, up to Lebesgue measure zero.
\end{definition}

{\bf Main results and the structure of the paper.}   In section 2, we turn to the study of general spectral subsets $\Omega$ of the line, i.e., $\br^k$ for $k = 1$. In Theorem \ref{th4.1}, we prove that if some set $\Omega$ is a first part in a spectral pair then there is a canonically associated unitary one-parameter group $U(t)$ consisting of local translations in $L^2(\Omega)$. If $\Omega$ is further assumed open, then this becomes a statement about the infinitesimal generator of  $U(t)$ as a self-adjoint extension of the minimal operator for $\Omega$.
   
Section 2 contains a number of additional detailed results. We highlight the following: in Theorem \ref{pr4.8}, for the general case of linear spectral sets $\Omega$, we offer a geometric representation of the associated unitary one-parameter group $U(t)$ of local translations in $L^2(\Omega)$: We show that this one-parameter group  $U(t)$  is unitarily equivalent to an induced representation of $(\br, +)$ in the sense of Mackey \cite{Ma61}.

In section 3, we turn to a detailed analysis of the set of all self-adjoint extensions of the minimal operator for a fixed bounded open linear set $\Omega$ written as a union of a finite number of components, see Definition \ref{def1.1}. We show in Theorem \ref{th1.7} that self-adjoint extensions of the minimal operator $\Ds$ correspond to unitary $n\times n$ matrices $B$. This follows the same pattern as the ones in \cite{JPTa,JPTb,JPTc}, but here the intervals are all finite. In addition, in Proposition \ref{pr1.10} we describe the reproducing kernel Hilbert space structure for the graph-inner product and offer formulas for the kernel functions. In Theorem \ref{th2.12} we describe the spectral decomposition of the self-adjoint extensions in terms of the unitary matrix $B$.

    Section 4 deals with  spectral sets which are finite unions of intervals. In Theorem \ref{th3.1} we show that some finite union of intervals  $\Omega$ in $\br$ is spectral if and only if there is a strongly continuous unitary one-parameter group $U(t)$ acting in the Hilbert space $L^2(\Omega)$ by pointwise translation inside $\Omega$, i.e., sending points $x$ in $\Omega$ to $x + t$  whenever both are in $\Omega$. This extends, in dimension one, results from \cite{Fug74,Ped87} (the results due to Fuglede and Pedersen are formulated in terms of the ``integrability property'' for $\Omega$, which is similar to our local translation property but the translations are made only with small enough numbers; also the {\it equivalence} between the spectral property and the integrability property is true only for connected sets).

In Theorem \ref{th3.4} and Corollary \ref{cor3.6}, for a given $\Omega$, assumed spectral, we characterize those self-adjoint extensions of the minimal operator which correspond to spectral pairs $(\Omega, \Lambda)$ and moreover, we give a formula for the corresponding spectrum $\Lambda$.
Our Theorems \ref{th3.9}, \ref{th3.14}, and \ref{th4.10} together offer a geometric properties of spectral sets $\Omega$.

  Finally in section 5, as an illustrating example, we specialize to the case when $\Omega$ is the disjoint union of two disjoint open intervals. While this may appear overly specialized, we stress that it is of significance in the above mentioned applications to Lax-Phillips scattering theory. Part of these results can be found in \cite{La01,JPTb}, but we include here more detailed description including one for the associated groups of local translations.

\section{General spectral subsets $\Omega$ of $\br$}

In this section we study the case of spectral pairs $(\Omega, \Lambda)$ for bounded Borel subsets of $\br$; a key ingredient is a result from \cite{BoMa11,IoKo12} that the spectrum $\Lambda$ has a finite period. But first we show that the spectral property implies the existence of a certain unitary group of local translations and we give a detailed description of this unitary group and various equivalent forms.

\begin{definition}\label{def3.3} 
Let $\Omega$ be a bounded Borel subset of $\br$. 
A {\it unitary group of local translations} on $\Omega$ is a strongly continuous one parameter unitary group $U(t)$ on $L^2(\Omega)$ with the property that for any $f\in L^2(\Omega)$ and any $t\in\br$,
\begin{equation}
(U(t)f)(x)=f(x+t)\mbox{ for a.e }x\in \Omega\cap(\Omega-t)
\label{eq3.3.1}
\end{equation}

If $\Omega$ is spectral with spectrum $\Lambda$, we define the Fourier transform $\mathcal F:L^2(\Omega)\rightarrow l^2(\Lambda)$ 
\begin{equation}\label{eq3.3.2}
\F f=\left(\ip{f}{\frac{1}{\sqrt{|\Omega|}}e_\lambda}\right)_{\lambda\in\Lambda},\quad(f\in L^2(\Omega)).
\end{equation}
We define the unitary group of local translations associated to $\Lambda$ by 
\begin{equation}
U_\Lambda(t)=\F^{-1}\hat U_\Lambda(t)\F\mbox{ where }\hat U_\Lambda(t)(a_\lambda)=(e^{2\pi i\lambda t}a_\lambda),\quad ((a_\lambda)\in l^2(\Lambda).
\label{eq3.3.3}
\end{equation}
\end{definition}

\begin{theorem}\label{th4.1}
Let $\Omega$ be a bounded Borel subset of $\br$. Assume that $\Omega$ is spectral with spectrum $\Lambda$. Let $U_\Lambda$ be the associated unitary group as in \eqref{eq3.3.3}. Then $U:=U_\Lambda$ is a unitary group of local translations.

\end{theorem}

\begin{proof}

 We will show that \eqref{eq3.3.1} holds. First note that $U(t)e_\lambda=e^{2\pi i\lambda t}e_\lambda$ for $\lambda\in\Lambda$ and $t\in\br$. So for $t\in\br$ and $x\in\Omega\cap(\Omega-t)$ we have

$$(U(t)e_\lambda)(x)=e^{2\pi i\lambda t}e^{2\pi i \lambda x}=e^{2\pi \lambda(x+t)}=e_\lambda(x+t),$$
hence \eqref{eq3.3.1} holds {\it everywhere} for $e_\lambda$.

 Let $f\in L^2(\Omega)$. Since the set $\{e_\lambda: \lambda\in\Lambda\}$ is an orthogonal basis for $L^2(\Omega)$, we can approximate $f$ in $L^2(\Omega)$ by a sequence of functions $f_n$ which are finite linear combinations of 
the functions $e_\lambda$, $\lambda\in\Lambda$. By passing to a subsequence we can also assume that $f_n(x)$ converges to $f(x)$ for $x\in\Omega\setminus E$ where $E$ has Lebesgue measure zero. Since $U(t)$ is unitary, we have that $U(t)f_n$ converges to $U(t)f$, and again by passing to a subsequence we can assume that $(U(t)f_n)(x)$ converges to $(U(t)f)(x)$ for $x\in\Omega\setminus F$ where $F$ has Lebesgue measure zero. 

The set $E\cup (E-t)$ has Lebesgue measure zero. Take $x\in (\Omega\cap(\Omega-t))\setminus(E\cup (E-t)\cup F)$. We have $f_n(x)\rightarrow f(x)$ and $f_n(x+t)\rightarrow f(x+t)$. Also 
$$(U(t)f_n)(x)=f_n(x+t)\rightarrow f(x+t)$$
At the same time $(U(t)f_n)(x)$ converges to $(U(t)f)(x)$. Thus we have $(U(t)f)(x)=f(x+t)$ for $x\in(\Omega\cap(\Omega-t))\setminus(E\cup (E-t)\cup F)$.

\end{proof}

Our next goal is to give a more precise description of the group of local translations associated to a spectrum. One of the main ingredients that we will use is the fact that any spectrum is {\it periodic} (see \cite{BoMa11,IoKo12}).

\begin{definition}\label{def3.8}
If $\Omega$, is spectral then any spectrum $\Lambda$ is {\it periodic} with some period $p\neq0$, i.e., $\Lambda+p=\Lambda$, and $p$ is an integer multiple of $\frac{1}{|\Omega|}$. We call $p$ a period for $\Lambda$. If $p=\frac{k(p)}{|\Omega|}$ with $k(p)\in\bn$, then $\Lambda$ has the form

\begin{equation}
\Lambda=\{\lambda_0,\dots,\lambda_{k(p)-1}\}+p\bz,
\label{eq3.8.1}
\end{equation}  
with $\lambda_0,\dots,\lambda_{k(p)-1}\in[0,p)$, see \cite{BoMa11,IoKo12}. The reason that there are $k(p)$ elements of $\Lambda$ in the interval $[0,p)$ can be seen also from the fact that the Beurling density of a spectrum $\Lambda$ has to be ${|\Omega|}$, see \cite{Lan67a}.

\end{definition}

\begin{remark}\label{rem4.1}
According to \cite{IoKo12}, if $\Omega$ has Lebesgue measure 1 and is spectral with spectrum $\Lambda$, with $0\in\Lambda$, then $\Lambda$ is periodic, the period $p$ is an integer and $\Lambda$ has the form
\begin{equation}
\Lambda=\{\lambda_0=0,\lambda_1,\dots,\lambda_{p-1}\}+p\bz,
\label{eq4.1}
\end{equation}
where $\lambda_i$ in $[0,p)$ are some distinct real numbers.
\end{remark}

We recall a few lemmas and propositions (see \cite{DuJo12b} and the references therein) that exploit the periodicity of the spectrum to give some information about the structure of $\Omega$. 
\begin{proposition}\label{pr3.9}
If $\Omega=\cup_{i=1}^n(\alpha_i,\beta_i)$, $\alpha_1<\beta_1<\alpha_2<\beta_2<\dots<\alpha_n<\beta_n$ is spectral, $p\in\bn$ and $\alpha_i,\beta_i\in\frac1p\bz$ for all $i=1,\dots,n$, then any spectrum $\Lambda$ for $\Omega$ has $p$ as a period. 
\end{proposition}

%\begin{proof}
%Let $\Lambda$ be a spectrum for $\Omega$. Take $\lambda\in\Lambda$ and assume $\lambda+p\not\in\Lambda$. We have, for $\lambda'\in\Lambda$, $\lambda'\neq \lambda$:
%$$\ip{e_{\lambda+p}}{e_{\lambda'}}_{L^2(\Omega)}=\sum_{i=1}^n\frac{1}{2\pi i(\lambda+p-\lambda')}(e^{2\pi i (\lambda+p-\lambda')\beta_i}-e^{2\pi i(\lambda+p-\lambda')\alpha_i})$$$$=\frac{1}{2\pi i(\lambda+p-\lambda')}\sum_{i=1}^n(e^{2\pi i (\lambda-\lambda')\beta_i}-e^{2\pi i(\lambda-\lambda')\alpha_i})=\frac{\lambda-\lambda'}{\lambda+p-\lambda'}\ip{e_\lambda}{e_{\lambda'}}_{L^2(\Omega)}=0$$
%Similarly $\ip{e_{\lambda+p}}{e_\lambda}_{L^2(\Omega)}=0$. But this would contradict the completeness of $\{e_\lambda : \lambda\in\Lambda\}$. 
%\end{proof}

\begin{corollary}\label{cor3.9}
If $\Omega=\cup_{i=1}^n(\alpha_i,\beta_i)$, $\alpha_1<\beta_1<\alpha_2<\beta_2<\dots<\alpha_n<\beta_n$ is spectral and $\alpha_i,\beta_i\in\bz$ for all $i=1,\dots,n$ and if a spectrum $\Lambda$ has period $\frac{k}{|\Omega|}$ then $k$ divides $|\Omega|$. 
\end{corollary}

%\begin{proof}
%$|\Omega|$ is an integer $N$. Then $\frac{1}{N}\Omega$ has measure 1, endpoints in $\frac1N\bz$ and spectrum $N\Lambda$ with period $k$. By Proposition \ref{pr3.9}, $k$ has to divide $N$. 
%\end{proof}
%
%The next proposition can be found in e.g. \cite{BoMa11,LaWa97,Ped96,IoKo12}. We include the proof for the benefit of the reader.
\begin{proposition}\label{pr4.15}
Let $\Omega$ be a bounded Borel set of measure 1. Assume that $\Omega$ is spectral with spectrum $\Lambda$, $0\in\Lambda$, which has period $p$. Then $\Omega$ is a $p$-tile of $\br$ by $\frac{1}{p}\bz$-translations, i.e., for almost every $x\in\br$, there exist exactly $p$ integers $j_1,\dots, j_{p}$ such that $x$ is in $\Omega+\frac{j_i}{p}$, $i=1,\dots,p$.

Also, for a.e. $x$ in $\Omega$, there are exactly $k$ integers $j_1,\dots,j_{p}$ such that $x+\frac{j_i}{p}$ is in $\Omega$ for all $i=1,\dots,p$.  

\end{proposition}
%
%\begin{proof}
%The statements follow if we prove that
%\begin{equation}
%\frac{1}{p}\sum_{j\in\bz}\chi_\Omega\left(x+\frac{j}{p}\right)=1\mbox{ a.e. on $[0,\frac1p]$}.
%\label{eq4.15.1}
%\end{equation}
%We can assume $\Lambda$ contains $0$ so $p\bz\subset\Lambda$. 
%We compute the Fourier coefficients: for $l\in\bz$, 
%$$p\int_0^{\frac1p}\frac{1}{p}\sum_{j\in\bz}\chi_{\Omega}\left(x+\frac jp\right)e_{-lp}(x)\,dx=\sum_{j\in\bz}\int_{\frac jp}^{\frac{j+1}{p}}\chi_{\Omega}(x)e_{-lp}(x)\,dx=\int_{\br}\chi_{\Omega}(x)e_{-lp}(x)\,dx=\delta_l.$$
%This implies \eqref{eq4.15.1}.
%
%\end{proof}

\begin{definition}\label{def4.16}
For $\varphi\in L^\infty(\br)$, define the multiplication operator $M(\varphi)$ on $L^2(\br)$ by $M(\varphi)f=\varphi f$, $f\in L^2(\Omega)$.
\end{definition}

\begin{lemma}\label{lem4.16}
Let $\Omega$ be a bounded Borel set of measure 1. Assume that $\Omega$ is spectral with spectrum $\Lambda$, $0\in\Lambda$ which has period $p$. Let $P(p\bz)$ be the orthogonal projection in $L^2(\Omega)$ onto the closed subspace spanned by $\{e_{kp} : k\in\bz\}$. Then, for $f\in L^2(\Omega)$,
\begin{equation}
(P(p\bz)f)(x)=\frac{1}{p}\sum_{j\in\bz}f\left(x+\frac jp\right),\mbox{ for a.e. $x\in\Omega$}.
\label{eq4.16.1}
\end{equation}
(We define $f$ to be zero outside $\Omega$)

\end{lemma}

%
%\begin{proof}
%
%Note first that the function on the right side of \eqref{eq4.16.1} has period $\frac1p$ and therefore it is in the $L^2$-span of the functions $e_{pl}$, $l\in\bz$. 
%We have, for $f\in L^2(\Omega)$, $l\in\bz$: 
%$$\int_{\Omega}\frac{1}{p}\sum_{j\in\bz}f\left(x+\frac jp\right)e_{-lp}(x)\,dx=\frac{1}{p}\sum_{j\in\bz}\int_{\br}\chi_\Omega\left(x-\frac jp\right)f(x)e_{-lp}(x)\,dx$$$$=\int f(x)e_{-lp}(x)=\ip{f}{e_{lp}}_{L^2(\Omega)}.$$
%
%Therefore, since $e_{lp}$ form an orthogonal basis for their span, equation \eqref{eq4.16.1} follows.
%
%\end{proof}

\begin{proposition}\label{pr4.17}
Let $\Omega$ be a bounded Borel set of measure 1. Assume that $\Omega$ is spectral with spectrum $\Lambda$, which has period $p$ and assume $0\in\Lambda$. Let $\Lambda=\{\lambda_0=0,\lambda_1,\dots,\lambda_{p-1}\}+p\bz$ with $\lambda_i\in[0,p)$, $i=0,\dots,p-1$. Then the projection $P(\lambda_i+p\bz)$ onto the span of $\{e_{\lambda_i+kp}: k\in\bz\}$ has the following formula: for $f\in L^2(\Omega)$,
\begin{equation}
(P(\lambda_i+p\bz)f)(x)=e_{\lambda_i}(x)\frac{1}{p}\sum_{j\in\bz}f\left(x+\frac jp\right)e_{-\lambda_i}\left(x+\frac jp\right),\mbox{ for a.e. $x\in\Omega$}.
\label{eq4.17.1}
\end{equation}
\end{proposition}
%
%\begin{proof}
%A simple check shows that for $\lambda\in\Lambda$ we have that the rank one projections $P(\lambda)$ onto $e_\lambda$ are related by the following formula:
%\begin{equation}
%P(\lambda)=M(e_\lambda)P(0)M(e_{-\lambda}).
%\label{eq4.17.2}
%\end{equation}
%Then, we obtain
%\begin{equation}
%P(\lambda_i+p\bz)=M(e_{\lambda_i})P(p\bz)M(e_{-\lambda_i}).
%\label{eq4.17.3}
%\end{equation}
%Equation \eqref{eq4.17.1} follows from \eqref{eq4.17.3} and \eqref{eq4.16.1}.
%\end{proof}

\begin{proposition}\label{pr4.19}
Let $\Lambda=\{\lambda_0=0,\lambda_1,\dots,\lambda_{k(p)-1}\}+p\bz$ be as in Proposition \ref{pr4.17}. For $x\in\br$, let $\Omega_x=\{j\in\bz : x+\frac jp\in\Omega\}$. Then $|\Omega_x|=p$ for a.e. $x\in \br$ and, for $i,i'=0,\dots,p-1$:
\begin{equation}
\frac{1}{p}\sum_{j\in\Omega_x}e^{2\pi i(\lambda_i-\lambda_{i'})\frac jp}=\delta_{ii'}\mbox{ for a.e. $x\in\Omega$}.
\label{eq4.19.1}
\end{equation}
In other words , the set $\{\lambda_0,\dots,\lambda_{p-1}\}$ is spectral and, for a.e. $x\in\br$, $\frac1{p}\Omega_x$ is a spectrum for it.

\end{proposition}

The next theorem exploits the structure of the spectral set $\Omega$ described in the previous statements, to explain the form of the group of local translations. 

\begin{theorem}\label{pr4.8}
Let $\Omega$ be a bounded Borel subset of $\br$ with $|\Omega|=1$. Let $p\in\bn$. Suppose $\Omega$ $p$-tiles $\br$ by $\frac1p\bz$. Then, for a.e. $x\in\br$ the set 
\begin{equation}
\Omega_x:=\left\{k\in\bz : x+\frac kp\in\Omega\right\}
\label{eq4.8.1}
\end{equation}
has exactly $p$ elements 
\begin{equation}
\Omega_x=\{k_0(x)<k_1(x)<\dots<k_{p-1}(x)\}.
\label{eq4.8.2}
\end{equation}
For almost every $x\in\br$ there exist unique $y\in [0,\frac1p)$ and $i\in\{0,\dots,p-1\}$ such that $y+\frac{k_i(y)}p=x$. 

The functions $k_i$ have the following property
\begin{equation}
k_i(x+\frac1p)=k_i(x)-1,\quad(x\in\br,i=0,\dots,p-1).
\label{eq4.8.2.1}
\end{equation}

Consider the space of $\frac1p$-periodic vector valued functions $L^2([0,\frac1p),\bc^p)$. 
The operator $W:L^2(\Omega)\rightarrow L^2([0,\frac1p),\bc^p)$ defined by
\begin{equation}
(Wf)(x)=\begin{pmatrix}f\left(x+\frac{k_{0}(x)}p\right)\\
\vdots\\
f\left(x+\frac{k_{p-1}(x)}p\right)
\end{pmatrix},\quad(x\in[0,\frac1p),f\in L^2(\Omega)),
\label{eq4.8.3}
\end{equation}
is an isometric isomorphism with inverse
\begin{equation}
W^{-1}\begin{pmatrix}
f_0\\
\vdots\\
f_{p-1}\end{pmatrix}(x)=f_i(y),\mbox{ if }x=y+\frac{k_i(y)}p,\mbox{ with }y\in[0,\frac1p), i\in\{0,\dots,p-1\}.
\label{eq4.8.4}
\end{equation}

A set $\Lambda$ of the form $\Lambda=\{0=\lambda_0,\lambda_1,\dots,\lambda_{p-1}\}+p\bz$ is a spectrum for $\Omega$ if and only if $\{\lambda_0,\dots,\lambda_{p-1}\}$ is a spectrum for $\frac1p\Omega_x$ for a.e. $x\in[0,\frac1p)$. 

The exponential functions are mapped by $W$ as follows:
\begin{equation}
(We_{\lambda_i+np})(x)=e_{\lambda_i+np}(x)\begin{pmatrix}e_{\lambda_i}(\frac{k_0(x)}p)\\
\vdots\\
e_{\lambda_i}(\frac{k_{p-1}(x)}p)
\end{pmatrix}=:F_{i,n}(x),\quad (i=0,\dots,p-1, n\in\bz,x\in[0,\frac1p)).
\label{eq4.8.5}
\end{equation}

For $x$ in $\br$ define the $p\times p$ unitary matrix $\M_x$ which has column vectors 
$$v_i(x):=\frac{1}{\sqrt{p}}(e_{\lambda_i}(\frac{k_0(x)}p),e_{\lambda_i}(\frac{k_1(x)}p),\dots,e_{\lambda_i}(\frac{k_{p-1}(x)}p))^t,\quad i=0,\dots,p-1.$$

Let $U_\Lambda$ be the group of local translations on $\Omega$ associated to a spectrum $\Lambda$. 
Consider the one-parameter unitary group $\U_p$ on $L^2([0,\frac1p),\bc^p)$ defined by 
\begin{equation}
(\U_p(t)F)(x)=\M_x\M_{x+t}^*F(x+t),\quad (x,t\in\br, F\in L^2([0,\frac1p),\bc^p)).
\label{eq4.8.6}
\end{equation}
Then $W$ intertwines $U_\Lambda$ and $\U_p$:
\begin{equation}
WU_\Lambda(t)=\U_p(t)W.
\label{eq4.8.7}
\end{equation}

\end{theorem}

\begin{proof}
The first statement follows from the fact that $\Omega$ $p$-tiles $\br$ by $\frac1p\bz$. The second statement follows from this and the fact that $[0,\frac1p)$ tiles $\br$ by $\frac1p\bz$.

To check that $W$ and $W^{-1}$, as defined, are inverse to each other requires just a simple computation. We verify that $W$ is isometric. 

For a subset $S$ of $\bz$ with $|S|=p$ define 
$$A_S:=\{x\in[0,\frac1p) : \Omega_x=S\}.$$
Note that, since $\Omega$ is bounded, $A_S=\ty$ for all but finitely many sets $S$. Also we have the following partition of $\Omega$. 
$$\bigcup_{|S|=p}(A_S+\frac1p S)=\Omega.$$

Take $f\in L^2(\Omega)$.  We have 
$$\|Wf\|_{L^2([0,\frac1p),\bc^p)}^2=\int_0^{\frac1p}\sum_{j=0}^{p-1}\left|f\left(x+\frac{k_j(x)}p\right)\right|^2\,dx=
\sum_{|S|=p}\int_{A_S}\sum_{j=0}^{p-1}\left|f\left(x+\frac{k_j(x)}p\right)\right|^2\,dx$$
$$=\sum_{|S|=p}\int_{A_S}\sum_{s\in S}\left|f\left(x+\frac{s}p\right)\right|^2\,dx=\sum_{|S|=p}\sum_{s\in S}\int_{A_S+\frac sp}|f(x)|^2\,dx=\int_{\Omega}|f(x)|^2\,dx.$$

Equation \eqref{eq4.8.5} requires just a simple computation. 

If $\Lambda$ is a spectrum for $\Omega$, then we saw in Proposition \ref{pr4.19} that $\frac1p\Omega_x$ has spectrum $\{\lambda_0,\dots,\lambda_{p-1}\}$ for a.e. $x\in\br$.

For the converse, if $\{\lambda_0,\dots,\lambda_{p-1}\}$ is a spectrum for a.e. $x\in[0,\frac1p)$, then for a.e. $x\in [0,\frac1p)$,the vectors $v_i(x)=\frac1{\sqrt{p}}(e_{\lambda_i}(\frac{k_0(x)}p),\dots,e_{\lambda_i}(\frac{k_{p-1}(x)}p))^t$, $i=0,\dots,p-1$, form an orthonormal basis for $\bc^p$. Then the functions $F_{i,n}$ in \eqref{eq4.8.5} form an orthonormal basis for $L^2([0,\frac1p),\bc^p)$ as can be seen by a short computation. To see that the functions $F_{i,n}$ span the entire Hilbert space, take $H$ in $L^2([0,\frac1p),\bc^p)$ such that $H\perp F_{i,n}$ for all $i=0,\dots,p-1$, $n\in\bz$. Then 
$$0=\int_0^{\frac1p}e_{\lambda_i+np}(x)\ip{H(x)}{v_i(x)}_{\bc^p}\,dx.$$
Since the functions $e_{np}$ are complete in $L^2[0,\frac1p)$ it follows that $e_{\lambda_i}(x)\ip{H(x)}{v_i(x)}_{\bc^p}=0$ for a.e. $x\in[0,\frac1p)$. So $\ip{H(x)}{v_i(x)}=0$ for a.e. $x\in[0,\frac1p)$ and all $i=0,\dots,p-1$. Then $H(x)=0$ for a.e. $x\in[0,\frac1p)$.

Next, we check that $\U_p$ is well defined, so the function in \eqref{eq4.8.6} is $\frac1p$-periodic. We have 

$$\M_{x+\frac{1}p}=\frac{1}{\sqrt p}\left(e_{\lambda_j}(\frac{k_i(x+\frac1p)}p)\right)_{i,j=0,\dots,p-1}=\frac{1}{\sqrt p} \left(e_{\lambda_j}(\frac{k_i(x)-1}{p})\right)_{i,j},$$
therefore $\M_{x+\frac1p}=\M_xD_\lambda(\frac1p)^*$, where $D_\lambda(\frac{1}{p})$ is the diagonal matrix with entries $e_{\lambda_i}(\frac1p)$.

We have, for $x,t\in\br$, $F\in L^2([0,\frac1p),\bc^p)$:
$$\M_{x+\frac1p}\M_{x+t+\frac{1}p}^*F(x+t+\frac1p)=\M_xD_\lambda(\frac1p)^* D_\lambda(\frac1p)\M_{x+t}^*F(x+t)=\M_x\M_{x+t}^*F(x+t).$$

The fact that the matrices $\M_x$ and $\M_{x+t}$ are unitary implies that $\U_p(t)$ is unitary.

To obtain \eqref{eq4.8.7}, it is enough to verify it on the basis $e_{\lambda_i+np}$ and that is equivalent to:
\begin{equation}
\U_p(t)F_{i,n}=e_{\lambda_i+np}(t)F_{i,n},\quad(i=0,\dots,p-1,n\in\bz).
\label{eq4.8.8}
\end{equation}

Note first that $\M_{x}^*v_i(x)=\delta_i$ for all $x\in\br$, $i=0,\dots,p-1$. We have 
$$(\U_p(t)F_{i,n})(x)=\M_x\M_{x+t}^*F_{i,n}(x+t)=\sqrt{p}e_{\lambda_i+np}(x+t)\M_x\M_{x+t}^*v_{i}(x+t)=\sqrt{p}e_{\lambda_i+np}(x+t)\M_x\delta_i$$$$=\sqrt{p}e_{\lambda_i+np}(t)e_{\lambda_i+np}(x)v_{i}(x)=e_{\lambda_i+np}(t)F_{i,n}(x).$$

\end{proof}

In the next two propositions we give some equivalent representations of the group of local translations, one involving the usual translation in $\br$ and the second involving induced representations in the sense of Mackey. 

\begin{proposition}\label{pr4.9}
Let $\Omega$ be a bounded Borel subset of $\br$ with $|\Omega|=1$. Assume that $\Omega$ has a spectrum $\Lambda$ with period $p\in\bn$ and 
$$\Lambda=\{\lambda_0=0,\lambda_1,\dots,\lambda_{p-1}\}+p\bz.$$
Identify $L^2[0,\frac1p]$ (with the normalized Lebesgue measure), with $\frac1p$-periodic $L^2$-functions. Define the projections $P(\lambda_i+p\bz)$ as in \eqref{eq4.17.1}, but for all $x\in\br$. Define the operator 
\begin{equation}
W:L^2(\Omega)\rightarrow\underbrace{L^2[0,\frac1p]\oplus\dots\oplus L^2[0,\frac1p]}_{\mbox{$p$ times }},\quad Wf=\left(e_{-\lambda_i}P(\lambda_i+p\bz)f\right)_{i=1,\dots,p},\quad(f\in L^2(\Omega)).
\label{eq4.9.1}
\end{equation}
Then $W$ is an isometric isomorphism with the following properties:
\begin{enumerate}
	\item for all $i=0,\dots,p-1$, $k\in\bz$,
	\begin{equation}
	We_{\lambda_i+pk}=(0,\dots,0,\underbrace{e_{pk}}_{\mbox{$i$-th position}},0,\dots,0)
\label{eq4.9.2}
\end{equation}

	\item Let $P_i$ be the projection onto the $i$-th component in $L^2[0,\frac1p]\oplus\dots\oplus L^2[0,\frac1p]$. Then 
	\begin{equation}
	WP(\lambda_i+p\bz)W^*=P_i,\quad i=1,\dots,p
\label{eq4.9.3}
\end{equation}
\item Let $U_\Lambda$ be the local translation group associated to $\Lambda$. Let $T_t$ be the translation operator on $L^2[0,\frac1p]$,
\begin{equation}
(T_tf)(x)=f(x+t),\quad(x,t\in\br).
\label{eq4.9.4}
\end{equation} 

Then $U_\Lambda$ commutes with the projections $P(\lambda_i+p\bz)$, $i=0,\dots,p-1$ and, for $t\in\br$, 
\begin{equation}
WU_\Lambda(t) W^*=\begin{bmatrix}
e_{\lambda_0}(t)T_t&0&\cdots&0\\
0&e_{\lambda_1}(t)T_t&\cdots &0\\
& &\ddots& \\
0&\cdots&0&e_{\lambda_{p-1}}(t)T_t
\end{bmatrix}
\label{eq4.9.5}
\end{equation}

\end{enumerate}

\end{proposition}

\begin{proof}
We check (i). Since $P(\lambda_i+p\bz)e_{\lambda_j+pk}=\delta_{ij}e_{\lambda_j+pk}$, \eqref{eq4.9.2} follows. Since $\Lambda$ is a spectrum, this implies that $W$ maps an orthonormal basis to an orthonormal basis, so it is an isometric isomorphism. 

(ii) can be checked on the basis $e_{\lambda_i+pk}$, $i=0,\dots,p-1$, $k\in\bz$. For (iii), we have $U_\Lambda(t)e_{\lambda_i+pk}=e_{\lambda_i}(t)e_{pk}(t)e_{\lambda_i+pk}$ and 
$T_te_{pk}=e_{pk}(t)e_{pk}$. This implies \eqref{eq4.9.5} and the fact that $U_\Lambda$ commutes with the projections $P(\lambda_i+p\bz)$. 

\end{proof}

\begin{proposition}\label{pr4.10}
Let $\Omega$ be a bounded Borel subset of $\br$ with $|\Omega|=1$. Assume that $\Omega$ has a spectrum $\Lambda$ with period $p\in\bn$ and 
$$\Lambda=\{\lambda_0=0,\lambda_1,\dots,\lambda_{p-1}\}+p\bz.$$
Let $U_\Lambda$ be the one-parameter group of local translations associated to $\Lambda$.

Let $T$ be a unitary matrix with eigenvalues $e_{\lambda_0}(\frac1p),\dots,e_{\lambda_{p-1}}(\frac1p)$. Consider the induced representation from $\bz$ to $\br$: let $\H_T$ be the Hilbert space:
\begin{equation}
\H_T:=\left\{ f:\br\rightarrow \bc^p : f\mbox{ measurable }, \int_0^{\frac1p}\|f(x)\|_{\bc^p}^2\,dx<\infty, f(x+\frac 1p)=T f(x)\mbox{ for a.e. }x\in\br\right\},
\label{eq4.10.1}
\end{equation}
with inner product
\begin{equation}
\ip{f}{g}_{\H_T}=p\int_0^{\frac1p}\ip{f(x)}{g(x)}_{\bc^p}\,dx.
\label{eq4.10.2}
\end{equation}
Define the one-parameter group of unitary transformations 
\begin{equation}
(U_T(t)f)(x)=f(x+t)\quad(t,x\in\br, f\in \H_T).
\label{eq4.10.3}
\end{equation}
Then there exists a isometric isomorphism from $L^2(\Omega)$ onto $\H_T$ that intertwines $U_\Lambda$ and $U_T$, i.e., 
\begin{equation}
WU_\Lambda(t)=U_T(t)W,\quad(t\in\br).
\label{eq4.10.4}
\end{equation}
\end{proposition}

\begin{proof}
Let $A$ be a $p\times p$ unitary matrix that diagonalizes $T$, i.e., $A^{-1}TA$ is the diagonal matrix with entries $e_{\lambda_k}(\frac{1}{p})$. Let $v_k=(A_{0,k},\dots,A_{p-1,k})^t$ be the column vectors of $A$, $k=0,\dots,p-1$. Note that $Tv_k=e_{\lambda_k}(\frac1p)v_k$ and $\{v_k\}$ form an orthonormal basis for $\bc^p$. 

Define
\begin{equation}
F_{k,n}(x)=e_{\lambda_k+np}(x)v_k,\quad(x\in\br, k=0,\dots,p-1,n\in\bz).
\label{eq4.10.5}
\end{equation}
We check that $F_{k,n}$ is in $\H_T$. For every $x\in\br$, 
$$F_{k,n}(x+\frac1p)=e_{\lambda_k+np}(x+\frac 1p)v_k=e_{\lambda_k+np}(x)e_{\lambda_k}(\frac1p)v_k=e_{\lambda_k+np}(x)Tv_k=TF_{k,n}(x).$$
Thus, $F_{k,n}\in\H_T$.

We prove that $\{F_{k,n}\}$ is an orthonormal basis for $\H_T$. We have
$$\ip{F_{k,n}}{F_{l,m}}_{\H_T}=p\int_0^{\frac1p}\ip{v_k}{v_l}_{\bc^p}\cj{e_{\lambda_k+np}(x)}e_{\lambda_l+mp}(x)\,dx=\delta_{kl}p\int_0^{\frac1p}\cj{e_{np}(x)}e_{mp}(x)\,dx=\delta_{kl}\delta_{mn}.$$

So $\{F_{k,n}\}$ are orthonormal in $\H_T$. We prove that they are complete. Let $H\in\H_T$ such that $H\perp F_{k,n}$ for all $k=0,\dots,p-1$, $n\in\bz$. Then 
$$0=\ip{H}{F_{k,n}}_{\H_T}=p\int_0^{\frac1p}\ip{H(x)}{e_{\lambda_k}(x)v_k}_{\bc^p}e_{np}(x)\,dx.$$
Then, since $e_{np}$ form an ONB in $L^2[0,\frac1p]$ we get that $\ip{H(x)}{e_{\lambda_k}(x)v_k}_{\bc^p}=0$ for a.e. $x\in[0,\frac1p]$. 

But, note that $$\ip{H(x+\frac1p)}{e_{\lambda_k}(x+\frac1p)v_k}_{\bc^p}=\ip{TH(x)}{Te_{\lambda_k}(x)v_k}_{\bc^p}=\ip{H(x)}{e_{\lambda_k}(x)v_k}_{\bc^p},$$
so the function $x\mapsto \ip{H(x)}{e_{\lambda_k}(x)v_k}_{\bc^p}$ is $\frac1p$-periodic. Therefore $\ip{H(x)}{e_{\lambda_k}(x)v_k}_{\bc^p}=0$ for a.e. $x\in\br$. 

Since $\{v_k\}$ form an ONB in $\bc^p$, we get that $H(x)=0$ for a.e. $x\in\br$. Thus $\{F_{k,n}\}$ is complete.

Define $W$ from $L^2(\Omega)$ onto $\H_T$ by $W(e_{\lambda_k+np})=F_{k,n}$ for all $k=0,\dots,p-1$, $n\in\bz$. Then $W$ extends linearly to an isometric isomorphism. 

We have, for $x\in\br$, $t\in\br$, $k=0,\dots,p-1$, $n\in\bz$: 
$$(WU_\Lambda(t)e_{\lambda_k+np})(x)=(W(e_{\lambda_k+np}(t)e_{\lambda_k+np}))(x)=e_{\lambda_k+np}(t)F_{k,n}(x)=
e_{\lambda_k+np}(t)e_{\lambda_k+np}(x)v_k$$$$
=e_{\lambda_k+np}(x+t)v_k=F_{k,n}(x+t)=(U_T(t)F_{k,n})(x)=(U_T(t)We_{\lambda_k+np})(x).
$$
This implies \eqref{eq4.10.4}.
\end{proof}

\section{Self-adjoint extensions of $\frac{1}{2\pi i}\frac{d}{dx}$}

In what follows we will restrict our attention to the case when $\Omega$ is a finite union of intervals. When such a union is spectral, the group of local translations has as infinitesimal generator an extension of the differential operator $\frac1{2\pi i}\frac{d}{dx}$.

 Let $\Omega$ be a bounded open subset in $\br^k$ of finite positive Lebesgue measure. In the general case of \cite{Fug74}, one deals with the partial derivative operators defined on the common dense domain  $C_0^\infty$  in $L^2(\Omega)$. We will refer to these as the minimal operators.

      If the dimension $k = 1$, the minimal operator will be denoted $\Ds$, or $\Dmin$  when the emphasis is needed. While the general framework for our analysis is the interplay between operator theoretic and spectral theoretic questions for commuting self-adjoint extension operators, we will make two restrictions here: one is $k = 1$, and the other is that we assume $\Omega$ is a finite union of connected components (intervals). In this case $\Dmin$ has von Neumann deficiency indices $(n, n)$ where $n$ is the number of intervals. If $k > 1$, each of the minimal operators has deficiency indices $(\infty, \infty)$.

  Our first result for the case when $\Omega$ is a finite union of $n$ components, i.e., $n$ open intervals, is Theorem \ref{th1.7}: we spell out a bijective correspondence between the set of all self-adjoint extensions of the minimal operator on the one hand, and on the other, elements in the group $U_n$ of all unitary $n \times n$ complex matrices. Given a self-adjoint extension $A$, the corresponding $n \times n$ $B$ in $U_n$ is acting between the boundary values computed for functions in the domain of $A$. We will say that $B$ is the associated boundary matrix.

      While this correspondence was also discussed in \cite{JPTa,JPTb,JPTc}, we have included the details here in a more complete form, as they will be needed. Moreover, our present treatment (see Proposition \ref{pr1.11}) includes explicit and constructive rules for the correspondence in both directions.

\begin{definition}\label{def1.1}

Let 
$$\Omega=\bigcup_{i=1}^n(\alpha_i,\beta_i)\mbox{ where } -\infty<\alpha_1<\beta_1<\alpha_2<\beta_2<\dots<\alpha_n<\beta_n<\infty.$$
So 
$$\Omega=\bigcup_{i=1}^nJ_i\mbox{ where }J_i=(\alpha_i,\beta_i)\mbox{ for all }i\in\{1,\dots,n\}.$$
On $\Omega$ we consider the Lebesgue measure $dx$. We denote by $\partial\Omega$ the boundary of $\Omega$, 
$$\partial\Omega=\left\{\alpha_i,\beta_i : i\in\{1,\dots,n\}\right\}.$$
For a function $f$ on $\partial\Omega$ we use the notation 
$$\int_{\partial \Omega}f=\sum_{i=1}^n(f(\beta_i)-f(\alpha_i)).$$

Consider the subspace of infinitely differentiable compactly supported functions $C_c^\infty(\Omega)$. 
We define the differential operator $\Ds$ on $C_c^\infty(\Omega)$:
$$\Ds f=\frac{1}{2\pi i}f',\quad(f\in C_c^\infty(\Omega))$$

Define also the subspace 
$$\D_0(\Omega):=\left\{f:\Omega\rightarrow\bc: f\mbox{ is absolutely continuous on each $J_i$, $f(\alpha_i+)=f(\beta_i-)=0$ for all $i$ and $f'\in L^2(\Omega)$ } \right\}.$$
\end{definition}

\begin{proposition}\label{pr1.2}
The operator $\Ds$ is symmetric. The adjoint $\Ds^*$ has domain 
\begin{equation}
\Dmax:=\operatorname*{domain}(\Ds^*)=\left\{ f:\Omega\rightarrow\bc : f\mbox{ is absolutely continuous on each interval $J_i$, $f'\in L^2(\Omega)$} \right\}
\label{eq1.2.1}
\end{equation}

and
\begin{equation}
\Ds^*f=\frac1{2\pi i }f'\mbox{ for }f\in\operatorname*{domain}(D^*).
\label{eq1.2.2}
\end{equation}
The operator $\Ds$ on $C_c^\infty(\Omega)$ has a closed extension to $\D_0(\Omega)$ and, for $f$ in $\D_0(\Omega)$, we also have $\Ds f=\frac1{2\pi i}f'$. The adjoint of the operator $\restr{\Ds}{\D_0(\Omega)}$ is the same as described above. The adjoint of $\Ds^*$ is
\begin{equation}
\left(\restr{\Ds}{C_c^\infty(\Omega)}\right)^{**}=\restr{\Ds}{\D_0(\Omega)}=\restr{\Ds^*}{\D_0(\Omega)}.
\label{eq1.2.3}
\end{equation} 
\end{proposition}

\begin{proof}
Let $f\in C_c^\infty(\Omega)$ and $g\in\domain(\Ds^*)$. Then $\ip{\Ds f}{g}=\ip{f}{\Ds^*g}$ which means that 
$$\frac{1}{2\pi i }\int_\Omega f'(x)\cj g(x)\,dx=\int f(x)\cj{\Ds^*g}(x)\,dx.$$
Define $\varphi(x)=\int_{\alpha_i}^x\Ds^*g(t)\,dt$ for all $x\in J_i$, $i\in\{1,\dots,n\}$. Then $\varphi$ is absolutely continuous and $\varphi'(x)=\Ds^*g(x)$ for a.e. $x\in\Omega$. Then, using an integration by parts we have
$$\frac{1}{2\pi i }\int_\Omega f'(x)\cj g(x)\,dx=\int_\Omega f(x)\cj{\varphi'}(x)\,dx=\int_{\partial\Omega}f\cj\varphi-\int_\Omega f'(x)\cj\varphi(x)\,dx=-\int_\Omega f'(x)\cj\varphi(x)\,dx,$$
since $f|_{\partial\Omega}=0$.

Then 
$$\int_\Omega f'(x)\left(-\frac{1}{2\pi i}g(x)+\varphi(x)\right)\,dx=0\mbox{ for all }f\in C_c^\infty(\Omega).$$
This means that the function $\varphi-\frac{1}{2\pi i }g$ is orthogonal to the range of the operator $\Ds$. 

Next, we find the orthogonal complement of the range of $\Ds$. Note that if $f\in C_c^\infty(\Omega)$ then $f|_{\partial\Omega}=0$ and $f(x)=\int_{\alpha_i}^xf'(t)\,dt$ for all $x\in J_i$, which implies that $\int_{\alpha_i}^{\beta_i}f'(x)\,dx=f(\beta_i)=0$. This means that every function which is constant on each interval $J_i$ is orthogonal to the range of $\Ds$.

Conversely, let $g\in L^2(\Omega)$ be orthogonal to the range of $\Ds$. Fix $i\in\{1,\dots,n\}$. 
If $f$ is in $C_c^\infty(J_i)$ and $\int_{\alpha_i}^{\beta_i}f(x)\,dx=0$, then define $\varphi(x)=2\pi i \int_{\alpha_i}^xf(t)\,dt$ for $x\in J_i$, zero otherwise. Then, $\varphi\in C_c^\infty(\Omega)$ and $\Ds\varphi=f$. 
So the range of $\Ds$ contains all functions in $C_c^\infty(J_i)$ with $\int_{\alpha_i}^{\beta_i}f(x)\,dx=0$. Any function $f\in L^2(\Omega)$ which is supported on $J_i$ and with $\int_{\alpha_i}^{\beta_i}f(t)\,dt=0$ can be approximated in $L^2(\Omega)$ by a sequence functions $f_n$ in $C_c^\infty(J_i)$ with $\int_{\alpha_i}^{\beta_i}f_n(t)\,dt=0$ and these are in the range of $\Ds$. Then, we have that $\ip{g}{f-\frac{1}{\beta_i-\alpha_i}\int_{\alpha_i}^{\beta_i}f(t)\,dt}=0$ for any $f\in L^2(\Omega)$ which is supported on $J_i$. Therefore
$$\int_\Omega g(x)\cj f(x)\,dx=\frac{1}{\beta_i-\alpha_i}\int_\Omega g(x)\,dx\cdot\int_\Omega \cj f(x)\,dx$$
and this implies that 
$g-\frac{1}{\beta_i-\alpha_i}\int_{\alpha_i}^{\beta_i}g(x)\,dx$ is orthogonal to all $f\in L^2(\Omega)$ which are supported on $J_i$. This proves that $g=\frac{1}{\beta_i-\alpha_i}\int_{\alpha_i}^{\beta_i}g(x)\,dx$ a.e. on $J_i$. 
Thus $g$ is constant on each interval $J_i$.

Returning to our computation of $\domain(\Ds^*)$ we obtain that $\varphi-\frac{1}{2\pi i}g$ is constant on each $J_i$. But then $g$ is absolutely continuous on each $J_i$ and $\frac{1}{2\pi i}g'=\varphi'=\Ds^*g$ a.e. on $\Omega$.

Conversely, if $g$ is absolutely continuous on the intervals $J_i$ and $g'\in L^2(\Omega)$ then the integration by parts used above shows that $\ip{\Ds f}{g}=\ip{f}{\frac{1}{2\pi i}g'}$, so $g\in\domain(\Ds^*)$ and $\Ds^*g=\frac{1}{2\pi i}g'$.

Next, we prove that the operator $\Ds$ on $\D_0(\Omega)$ is closed. Take $f_n$ in $\D_0(\Omega)$ which converges to some $f$ in $L^2(\Omega)$ and such that $f_n'$ converges to $g$ in $L^2(\Omega)$. Define 
$\varphi(x)=\int_{\alpha_i}^x g(t)\,dt=\ip{g}{\chi_{(\alpha_i,x)}}$ for $x\in J_i$. Then for all $x\in\Omega$, $\varphi(x)$ is the limit of $\ip{f_n}{\chi_{(\alpha_i,x)}}=\int_{\alpha_i}^xf_n(t)=f_n(x)-f_n(\alpha_i+)=f_n(x)$. 
Since $f_n$ converges to $f$ in $L^2(\Omega)$ we get that $f=\varphi$ a.e. on $\Omega$. This implies that $f$ is absolutely continuous and $f'=\varphi'=g$ a.e. on $\Omega$. Since $\varphi(\alpha_i+)=0$ it follows that $f(\alpha_i+)=0$. We also have 
$$f(\beta_i-)=\varphi(\beta_i-)=\int_{\alpha_i}^{\beta_i}g(t)\,dt=\lim_n \int_{\alpha_i}^{\beta_i}f_n'(t)\,dt=\lim_n(f_n(\beta_i-)-f(\alpha_i+))=0.$$

To prove that the adjoint of $\restr{\Ds}{\D_0(\Omega)}$ is as before, the same arguments can be used. Since $A=\restr{\Ds}{\D_0(\Omega)}$ is closed, $A^{**}=A$ \cite[Corollary 1.8, page 305]{Con90}.

\end{proof}

\begin{definition}\label{def1.3}
The deficiency spaces are defined by 
$$\mathcal L_+=\operatorname*{ker}(\Ds^*-iI)=\operatorname*{ran}(\Ds+iI)^{\perp},\quad \mathcal L_-=\operatorname*{ker}(\Ds^*+iI)=\operatorname*{ran}(\Ds-iI)^{\perp}.$$
The deficiency indices of $\Ds$ are $n_{\pm}=\operatorname*{dim}\mathcal L_{\pm}$.
\end{definition}

\begin{proposition}\label{pr1.4}
Define the numbers 
$$\gamma_i^+=\sqrt{\frac{4\pi}{e^{-4\pi \alpha_i}-e^{-4\pi \beta_i}}},\quad \gamma_i^-=\sqrt{\frac{4\pi}{e^{4\pi \beta_i}-e^{4\pi \alpha_i}}}$$
The functions $\{\gamma_i^+e^{-2\pi t}\chi_{J_i}(t) : i\in\{1,\dots,n\}\}$ form an orthonormal basis for $\mathcal L_+$ and the functions $\{\gamma_i^-e^{2\pi t}\chi_{J_i}(t) : i\in\{1,\dots,n\}\}$ form an orthonormal basis for $\mathcal L_-$. Therefore the deficiency indices are both equal to $n$, $n_+=n_-=n$. 

Every function $f$ in $\Dmax$ can be written uniquely as 
$$f(t)=f_0(t)+\sum_{i=1}^nc_i e^{-2\pi t}\chi_{J_i}(t)+\sum_{i=1}^nd_ie^{2\pi t}\chi_{J_i}(t),\quad(t\in\Omega)$$
for some $f_0\in\D_0(\Omega)$, $c_i,d_i\in \bc$, $i\in\{1,\dots,n\}$. 
\end{proposition}

\begin{proof}
We have $f\in \mathcal L_+$ iff $f$ is absolutely continuous with $f'\in L^2(\Omega)$ and $\frac{1}{2\pi i}f'=if$ on $\Omega$. Solving the differential equation on each interval $J_i$, we obtain that $f$ must be of the form 
$$f(t)=\sum_{i=1}^nc_ie^{-2\pi t}\chi_{J_i}(t),\quad(t\in\Omega)$$
for some constants $c_i\in\bc$. The functions $e^{-2\pi t}\chi_{J_i}(t)$ are obviously orthogonal and their $L^2$-norms are computed to be $1/\gamma_i^+$. This implies the first statement for $\mathcal L_+$ and $\mathcal L_-$ can be analyzed similarly. 

The last statement follows from the general theory (see \cite{DuSc2}): any element $f$ in the domain of $\Ds^*$ can be written as $f=f_0+f_++f_-$ with $f$ in the domain of the closed symmetric operator $\restr{\Ds}{\D_0(\Omega)}$, $f_+$ in $\mathcal L_+$ and $f_-$ in $\mathcal L_-$. It can be proved also directly, by arranging the constants $c_i$ and $d_i$ so that the values at the endpoints $\alpha_i$ and $\beta_i$ are matched. 
\end{proof}

\begin{theorem}\label{th1.5}
For every self-adjoint extension $A$ of $\restr{\Ds}{C_c^\infty(\Omega)}$ there exists a unique unitary $W$ from $\mathcal L_+$ to $\mathcal L_-$ such that $A$ is the restriction of $\restr{\Ds^*}{\Dmax}$ to 
\begin{equation}
\D_W=\left\{f+f_++Wf_+ : f\in\D_0(\Omega) , f_+\in\mathcal L_+\right\}.
\label{eq1.5.1}
\end{equation}
Conversely, for any unitary $W$ from $\mathcal L_+$ onto $\mathcal L_-$ the restriction of $\restr{\Ds}{\Dmax}$ to $\D_W$ is a self-adjoint extension of $\restr{\Ds}{C_c^\infty(\Omega)}$.
\end{theorem}

\begin{proof}
Since the closure of $\restr{\Ds}{C_c^\infty(\Omega)}$ is the closed symmetric operator $\restr{\Ds}{\D_0(\Omega)}$, any self-adjoint extension of the first operator will be an extension of the latter. Then, the rest follows from \cite[Theorem 2.17 page 321, Theorem 2.20 page 320]{Con90}.
\end{proof}

\begin{definition}\label{def1.6}
For a function $f$ in $\Dmax$ we define the vectors 
$$f(\alpha)=\left(\begin{array}{c}
f(\alpha_1)\\
f(\alpha_2)\\
\vdots\\
f(\alpha_n)
\end{array}\right),
\quad
f(\beta)=\left(\begin{array}{c}
f(\beta_1)\\
f(\beta_2)\\
\vdots\\
f(\beta_n)
\end{array}\right)
$$
\end{definition}

\begin{theorem}\label{th1.7}
For every self-adjoint extension $A$ of  $\restr{\Ds}{C_c^\infty(\Omega)}$ there exist a unique $n\times n$ unitary matrix $B$ such that $A$ is the restriction of $\restr{\Ds^*}{\Dmax}$ to 
\begin{equation}
\D_B:=\left\{f\in\Dmax : Bf(\alpha)=f(\beta)\right\}.
\label{eq1.7.1}
\end{equation}
Conversely for any unitary matrix $B$, the restriction of $\restr{\Ds}{\Dmax}$ to $\D_B$ is a self-adjoint extension of $\restr{\Ds}{C_c^\infty(\Omega)}$.

\end{theorem}

\begin{proof}
Before we prove the theorem we recall some definitions, see \cite[Definition 20 page 1234, Definition 25 page 1235]{DuSc2}.

\begin{definition}\label{def1.8}
A {\it boundary value} for the operator $\Ds$ is a continuous linear functional on the Hilbert space $\Dmax$ with the {\it graph inner product} 
$$\ip{f}{g}_{\gr}=\ip{f}{g}+\ip{\Ds f}{\Ds g},\quad(f,g\in \Dmax),$$
which vanishes on $\D_0$. If $\varphi$ is a boundary value for $\Ds$ then the equation $\varphi(f)=0$ is called a {\it boundary condition}. A set of boundary conditions $\varphi_i(f)=0$, $i=1,\dots,k$, is said to be {\it linearly independent} if the boundary values $\varphi_1,\dots, \varphi_k$ are linearly independent. A set of boundary conditions $\varphi_i(f)=0$, $i=1,\dots, k$, is said to be {\it symmetric} if, for $f,g\in\Dmax$, the equations $\varphi_i(f)=\varphi_i(g)=0$, $i=1,\dots,k$ imply the equation 
\begin{equation}
\ip{\Ds^*f}{g}=\ip{f}{\Ds^*g}.
\label{eq1.8.2}
\end{equation}
\end{definition}

We start with a lemma.

\begin{lemma}\label{lem1.9}
The linear functionals $f\mapsto f(\alpha_i)$, $f\mapsto f(\beta_i)$, $i=1,\dots, n$, $f\in\Dmax$, are linearly independent boundary values and span the space of boundary values.
\end{lemma}

\begin{proof}
First we have to make sure these functions are well defined, i.e., $f(\alpha_i)$ and $f(\beta_i)$ are well defined. Since $f$ is absolutely continuous, it is also uniformly continuous, therefore $\lim_{x\downarrow\alpha_i}f(x)$ exists and defines $f(\alpha_i)$. Similarly for $\beta_i$. 

Let $f_n\in\Dmax$, $f_n\rightarrow f\in\Dmax$ in the graph inner product. This means that $f_n\rightarrow f$ and $f_n'\rightarrow f'$ in $L^2(\Omega)$. Then for $x\in J_i$,
\begin{equation}
f_n(x)-f_n(\alpha_i)=\int_{\alpha_i}^xf_n'(x)\,dx=\ip{f_n'}{\chi_{(\alpha_i,x)}}\rightarrow\ip{f'}{\chi_{(\alpha_i,x)}}=f(x)-f(\alpha_i).
\label{eq1.9.1}
\end{equation}

Assume that $f_n(\alpha_i)$ does not converge to $f(\alpha_i)$. Taking a subsequence, we can assume that $f_n(\alpha_i)$ is bounded away from $f(\alpha_i)$. Since $f_n\rightarrow f$ in $L^2(\Omega)$ we can extract a subsequence which is convergent pointwise a.e. to $f$. Then take a point $x$ where this convergence holds and plug it in \eqref{eq1.9.1}. It follows that a subsequence of $f_n(\alpha_i)$ converges to $f(\alpha_i)$. The contradiction implies that $f_n(\alpha_i)$ converges to $f(\alpha_i)$. Similarly for $\beta_i$. So these functionals are continuous w.r.t the graph inner product. Obviously, they vanish on $\D_0$. 

To see that these boundary values are independent, we use \cite[Lemma 21, page 1234]{DuSc2}. Pick functions $f_i^+$ in $\Dmax$ with $f_i^+(\alpha_j)=\delta_{ij}$,  $f_i^+(\beta_j)=0$ and 
$f_i^-(\alpha_j)=0$, $f_i^-(\beta_j)=\delta_{ij}$, $i,j=1,\dots,n$. Then the $2n\times 2n$ matrix $(f_i^{\pm}(\alpha_i\mbox{ or }\beta_i))$ is a permutation of the identity matrix, so non-singular. 

The space of boundary values has dimension $n_++n_-=2n$ (\cite[Lemma21, page 1234]{DuSc2}), so these boundary values span the space of all boundary values.
\end{proof}

By \cite[Theorem 30, page 1238]{DuSc2} the self-adjoint extensions are the restrictions of $\Ds^*$ to the subspace of $\Dmax$ determined by a symmetric family of $n$ linearly independent condition. 

Given a unitary matrix $B$ we have to check the equations
$$A_i(f)=f(\beta_i)-\sum_{j=1}^nB_{ij}f(\alpha_j)=0,\quad i=1,\dots,n$$
form a symmetric family of linearly independent boundary conditions.

The fact that the linear functionals $A_i$ are boundary values follows from Lemma \ref{lem1.9}. To check that these are linearly independent pick $f_i\in\D_0$ with $f_i(\alpha_j)=\delta_{ij}$, $f(\beta_i)=0$, $i,j=1,\dots,n$. Then 
$A_i(f_j)=-B_{ij}$. So the matrix $(A_i(f_j))_{ij}$ is equal to $-B$, so it is not singular, therefore $A_i$, $i=1,\dots, n$ are linearly independent \cite[Lemma 21, page 124]{DuSc2}. To see that the family of boundary condition is symmetric we use the following Lemma:
\begin{lemma}\label{lem1.10}
For $f,g\in\Dmax$ 
$$2\pi i(\ip{\Ds^*f}{g}-\ip{f}{\Ds^*g})=\ip{f(\beta)}{g(\beta)}_{\bc^n}-\ip{f(\alpha)}{g(\alpha)}_{\bc^n}.$$
\end{lemma}

\begin{proof}
We use integration by parts:
$$2\pi i(\ip{\Ds^*f}{g}-\ip{f}{\Ds^*g})=\sum_{i=1}^n\left(\int_{\alpha_i}^{\beta_i}f'(x)\cj g(x)\,dx+\int_{\alpha_i}^{\beta_i}f(x)\cj g'(x)\,dx\right)$$
$$=\sum_{i=1}^n\left(\int_{\alpha_i}^{\beta_i}f'(x)\cj g(x)\,dx+f(\beta_i)\cj g(\beta_i)-f(\alpha_i)\cj g(\alpha_i)-\int_{\alpha_i}^{\beta_i}f'(x)\cj g(x)\,dx\right)=\ip{f(\beta)}{g(\beta)}_{\bc^n}-\ip{f(\alpha)}{g(\alpha)}_{\bc^n}.$$
\end{proof}

If we have $A_i(f)=A_i(g)=0$ then $f(\beta)=Bf(\alpha)$ and $g(\beta)=Bg(\alpha)$. Since $B$ is unitary this implies that $\ip{f(\alpha)}{g(\alpha)}_{\bc^n}=\ip{f(\beta)}{g(\beta)}_{\bc^n}$ and with Lemma \ref{lem1.10} this implies that $\ip{\Ds^*f}{g}=\ip{f}{\Ds^*f}$, so the family of boundary conditions is symmetric.

Conversely, if we have a self-adjoint extension, then this is the restriction of $\Ds^*$ to the subspace of $\Dmax$ determined by a symmetric family of $n$ linearly independent boundary conditions $A_i(f)=0$, $i=1,\dots,n$. 

Since the space of boundary values is spanned by the evaluation functionals there exist $c_{ij},d_{ij}\in\bc$ such that for all $f\in \Dmax$,
$$A_i(f)=\sum_{j=1}^nc_{ij}f(\alpha_j)+\sum_{j=1}^nd_{ij}f(\beta_j),\quad i=1,\dots,n.$$

Denote by $A(f)$ the vector 

$$A(f)=\begin{pmatrix}
A_1(f)\\
A_2(f)\\
\vdots\\
A_n(f)
\end{pmatrix}$$
Then $A(f)=Cf(\alpha)+Df(\alpha)$. We prove that $C$ and $D$ are non-singular. Suppose $Cv=0$ for some $v\in\bc^n$. Pick $f$ in $\Dmax$ such that $f(\alpha)=v$ and $f(\beta)=0$. Then $A(f)=Cv+D\cdot0=0$. Since the boundary conditions are symmetric, this implies, using Lemma \ref{lem1.10}, that $\ip{f(\alpha)}{f(\alpha)}=\ip{f(\beta)}{f(\beta)}$ so $v=0$. Thus $C$ is non-singular. Similarly $D$ is non-singular. 

The boundary conditions are equivalent to $f(\beta)=-D^{-1}Cf(\alpha)$. We prove that the matrix $-D^{-1}C$ is unitary. Take $v_1,v_2\in\bc^n$. Pick functions $f,g\in \Dmax$ with $f(\alpha)=v_1$, $f(\beta)=-D^{-1}Cv_1$, 
$g(\alpha)=v_2$, $g(\beta)=-D^{-1}Cv_2$. Then $f$ and $g$ satisfy the boundary conditions and, by their symmetry and Lemma \ref{lem1.10}, we must have  $\ip{f(\alpha)}{g(\alpha)}=\ip{f(\beta)}{g(\beta)}$ which means that 
$\ip{v_1}{v_2}=\ip{-D^{-1}Cv_1}{-D^{-1}Cv_2}$. Since $v_1,v_2$ were arbitrary this means that $B:=-D^{-1}C$ is unitary.

\end{proof}

\begin{proposition}\label{pr2.10}
Let $A$ be a self-adjoint extension of $\Ds$ and let $B$ be the associated unitary boundary matrix as in Theorem \ref{th1.7}. Suppose there exist $\lambda_1,\dots,\lambda_n$ in $\br$ such that the functions $e_{\lambda_i}$ are in $\domain(A)$, $i=1,\dots,n$. Define the matrices 
\begin{equation}
M_{\alpha}=(e_{\lambda_j}(\alpha_i))_{i,j=1}^n,\quad M_\beta=(e_{\lambda_j}(\beta_i))_{i,j=1}^n.
\label{eq2.10.1.1}
\end{equation}

Then 
\begin{equation}
BM_\alpha=M_\beta.
\label{eq2.10.1.2}
\end{equation}
The set $\{e_{\lambda_j}\}_{j=1}^n$ is orthogonal in $L^2(\Omega)$.
\end{proposition}

\begin{proof}
Equation \eqref{eq2.10.1.2} follows from \eqref{eq1.7.1} applied to the functions $e_{\lambda_j}$. For $i\neq j$ we have, with Lemma \ref{lem1.10},
$$0=\ip{Be_{\lambda_i}(\alpha)}{Be_{\lambda_j}(\alpha)}_{\bc^n}-\ip{e_{\lambda_i}(\alpha)}{e_{\lambda_j}(\alpha)}_{\bc^n}=\ip{e_{\lambda_i}(\beta)}{e_{\lambda_j}(\beta)}_{\bc^n}-\ip{e_{\lambda_i}(\alpha)}{e_{\lambda_j}(\alpha)}_{\bc^n}$$
$$=2\pi i(\ip{\Ds^*e_{\lambda_i}}{e_{\lambda_j}}_{L^2(\Omega)}-\ip{e_{\lambda_i}}{\Ds^* e_{\lambda_j}}_{L^2(\Omega)})=2\pi i(\lambda_i-\lambda_j)\ip{e_{\lambda_i}}{e_{\lambda_j}}_{L^2(\Omega)}.$$
\end{proof}

\begin{proposition}\label{pr1.10}
The vector space $\Dmax$ with the inner product 
\begin{equation}
\ip{f}{g}_{\gr}=\ip{f}{g}_{L^2(\Omega)}+\ip{\Ds f}{\Ds g}_{L^2(\Omega)},\quad(f,g\in \Dmax)
\label{eq1.10.1}
\end{equation}
(called the graph inner product)
is a reproducing kernel Hilbert space, so for every $x\in\cj\Omega$, there exist a unique function $k_x$ in $\Dmax$ such that 
\begin{equation}
\ip{f}{k_x}_{\gr}=f(x)
\label{eq1.10.2}
\end{equation}
for all $f\in\Dmax$.

For the endpoints we have 
\begin{equation}
k_{\alpha_i}(t)=\chi_{J_i}(t)\frac{\cosh2\pi(\beta_i-t)}{\sinh2\pi(\beta_i-\alpha_i)},\quad
k_{\beta_i}(t)=\chi_{J_i}(t)\frac{\cosh2\pi(t-\alpha_i)}{\sinh2\pi(\beta_i-\alpha_i)},
\label{eq1.10.3}
\end{equation}
where $\sinh$ and $\cosh$ are the sine and cosine hyperbolic functions.

For interior points $x\in\Omega$, if $x\in J_j$, then 
\begin{equation}
k_x(t)=A_x\chi_{(\alpha_j,x]}(t)\frac{\cosh2\pi(t-\alpha_j)}{\sinh2\pi(x-\alpha_j)}+B_x\chi_{[x,\beta_j)}\frac{\cosh2\pi(\beta_j-t)}{\sinh2\pi(\beta_j-x)},
\label{eq1.10.4}
\end{equation}
where $A_x, B_x$ are the solutions of the system of equations
\begin{equation}
A_x+B_x=1,\quad A_x\coth2\pi(x-\alpha_j)-B_x\coth2\pi(\beta_j-x)=0.
\label{eq1.10.5}
\end{equation}
\end{proposition}

\begin{proof}
Since $\Ds$ is a closed operator on $\Dmax$, we do have a Hilbert space. We proved in Lemma \ref{lem1.9} that the functionals $f\mapsto f(\alpha_i)$ are continuous with respect to the graph-inner product. The same argument shows that, for any $x\in\cj\Omega$, $\Dmax\ni f\mapsto f(x)$ is a continuous linear functional. Note, that by absolute continuity we can extend $f\in\Dmax$ to $\cj\Omega$ (see again the proof of Lemma \ref{lem1.9}). By Riesz' lemma, there exists a unique $k_x$ in $\Dmax$ such that $f(x)=\ip{f}{k_x}$ for all $f\in\Dmax$. 

For \eqref{eq1.10.3}, note that if $k_{\alpha_i}$ is given by this formula then:
$k_{\alpha_i}''(t)=4\pi^2 k_{\alpha_i}$ and $k_{\alpha_i}'(\alpha_i)=-1$, $k_{\alpha_i}'(\beta_i)=0$. Then, using integration by parts we have, for $f\in\Dmax$: 
$$\ip{f}{k_{\alpha_i}}_{\gr}=\int_{\alpha_i}^{\beta_i}f(t)\left(k_{\alpha_i}(t)-\frac{1}{4\pi^2}k_{\alpha_i}''(t)\right)\,dt +f(\beta_i)k_{\alpha_i}(\beta_i)-f(\alpha_i)k_{\alpha_i}(\alpha_i)=f(\alpha_i).$$
By uniqueness this proves that the repreoducing kernel function $k_{\alpha_i}$ has the formula in \eqref{eq1.10.3}. Similarly for $\beta_i$. 

For $x\in J_j$, note that $k_x(x-)=A_x\coth(x-\alpha_j)$, $k_x(x+)=B_x\coth(\beta_j-x)$. The second condition in \eqref{eq1.10.5} guarantees that $k_x$ is continuous at $x$. The function $k_x$ is $C^\infty$ on $(\alpha_j,x)$, $k_x'(\alpha_j)=0$, $k_x'(\beta_j)=0$ and on $(x,\beta_j)$ and $k_x'(x-)=A_x$, $k_x'(x+)=-B_x$. In particular $f\in \Dmax$. Then, for $f\in \Dmax$, using integration by parts as before, we have:
$$\ip{f}{k_x}_{\gr}=fk_x' |_{\alpha_j}^x+\, |_x^{\beta_j}=f(x)(k_x'(x-)-k_x'(x+))=f(x)(A_x+B_x)=f(x).$$

\end{proof}

\begin{remark}\label{rem2.11}
We note that the introduction of reproducing kernels in finite-order Sobolev spaces is of interest in mathematics of computations. As illustrated in for example \cite{FO10, Su10}, such kernels may serve as spline functions in numerical interpolations.
\end{remark}

\begin{proposition}\label{pr1.11}
Let $\Gamma_+$, $\Gamma_-$ be the $n\times n$ diagonal matrices with entries $(\gamma_i^+)$ and $(\gamma_i^-)$ respectively. For a vector $v=(v_1,\dots,v_n)$ let $E(v)$ be the diagonal matrix with entries $(v_i)$. The bijective correspondence between the isometries $W:\mathcal L_+\rightarrow \mathcal L_-$ (written as matrices in the orthonormal bases given in Proposition \ref{pr1.4}) and the unitary matrices $B$ in Theorem \ref{th1.7} is given by 
\begin{equation}
B=\left(\Gamma_- E(2\pi\beta)W+\Gamma_+E(-2\pi\beta)\right)\left(\Gamma_- E(2\pi\alpha)W+\Gamma_+E(-2\pi\alpha)\right)^{-1}.
\label{eq1.11.1}
\end{equation}

\end{proposition}

\begin{proof}
The domain determined by the isometry $W$ is $\D_W=\{f+f_++Wf_+: f\in \D_0(\Omega), f_+\in\mathcal L_+\}$ and it must coincide with the domain determined by $B$, which is $\D_B=\{f\in\Dmax: f(\beta)=Bf(\alpha)\}$. 
For any $c_1,\dots,c_n\in \bc$, consider the function 
$$f(t)=\sum_{i=1}^nc_i\gamma_i^+e^{-2\pi t}\chi_{J_i}(t)+\sum_{i=1}^n\sum_{j=1}^nW_{ij}c_j\gamma_i^-e^{2\pi t}\chi_{J_i}(t).$$
This is in the domain $\D_W=\D_B$.
Plug in $t=\alpha_i$ and $t=\beta_i$ to obtain 
\begin{equation}
f(\alpha)=(\Gamma_-E(2\pi\alpha)W+\Gamma_+E(-2\pi\alpha))c,\quad f(\beta)=(\Gamma_-E(2\pi\beta)W+\Gamma_+E(-2\pi\beta))c
\label{eq1.11.2}
\end{equation}

and we must have $Bf(\alpha)=f(\beta)$.

We claim that $\Gamma_-E(2\pi\alpha)W+\Gamma_+E(-2\pi\alpha)$ is invertible. We have
$$\Gamma_-E(2\pi\alpha)W+\Gamma_+E(-2\pi\alpha)=\Gamma_-E(2\pi\alpha)(W+E(-2\pi\alpha)\Gamma_-^{-1}\Gamma_+E(-2\pi\alpha)).$$
The matrix $E(-2\pi\alpha)\Gamma_-^{-1}\Gamma_+E(-2\pi\alpha)$ is a diagonal matrix with entries $\frac{\gamma_i^+}{\gamma_i^-}e^{-4\pi\alpha_i}=e^{2\pi (\beta_i-\alpha_i)}>1$. 
If $(W+E(-2\pi\alpha)\Gamma_-^{-1}\Gamma_+E(-2\pi\alpha))v=0$ for some $v\neq 0$, then $Wv=-E(-2\pi\alpha)\Gamma_-^{-1}\Gamma_+E(-2\pi\alpha)v$ so $$\|v\|=\|Wv\|=\|-E(-2\pi\alpha)\Gamma_-^{-1}\Gamma_+E(-2\pi\alpha)v\|>\|v\|.$$
The contradiction proves the claim.

Since this matrix is invertible, from \eqref{eq1.11.2} we obtain that $B$ is given by \eqref{eq1.11.1}.

\end{proof}

\begin{theorem}\label{th2.12}
Let $A$ be a self-adjoint extension of the operator $\restr{\Ds}{C_c^\infty(\Omega)}$. Let $B$ be the unitary $n\times n$ matrix associated to $A$ as in Theorem \ref{th1.7}. Consider the spectral measure $P$ for the operator $A$, $$A=\int_{\br}t\,dP(t).$$

Then the spectral measure $P$ is atomic, supported on the set 
\begin{equation}
\Lambda_B:=\left\{\lambda\in\bc : \det(I-E(-2\pi i\lambda\beta)BE(2\pi i\lambda\alpha))=0\right\},
\label{eq2.12.1}
\end{equation}
which is also the spectrum $\sigma(A)$ of $A$. 
For $\lambda\in\Lambda_B$, the eigenspace $P(\lambda)L^2(\Omega)$ is finite dimensional and 
$$P(\lambda)=\left\{\sum_{i=1}^nc_i\chi_{J_i}(t)e^{2\pi i\lambda t} : (I-E(-2\pi i\lambda\beta)BE(2\pi i\lambda\alpha))c=0\right\}.$$

\end{theorem}

\begin{proof}
Consider the operator 
$$K:=(A-iI)^{-1}=\int_{\br}\frac{1}{t-i}\,dP(t).$$
Since the function $t\mapsto\frac{1}{t-i}$ is bounded we have that $K$ is a bounded operator and it is indeed the inverse of $A-iI$. 
We claim that $K$ is a compact operator. Take $z_k\in L^2(\Omega)$ with $\|z_k\|_{L^2(\Omega)}\leq 1$ for all $k\in\bn$. Let $f_k:=Kz_k=(A-iI)^{-1}z_k$, so $f_k\in\Dmax$ and $z_k=Af_k-if_k=\frac{1}{2\pi i}f_k'-f_k$.

Since $(z_k)$ is bounded in $L^2(\Omega)$ and $K$ is a bounded operator $(f_k)$ is bounded in $L^2(\Omega)$, hence $(f_k')$ is also bounded in $L^2(\Omega)$ by some constant $M>0$. 

Take two points $x<y$ in one of the intervals $J_i$. We have, using the Schwarz inequality,
\begin{equation}
|f_k(y)-f_k(x)|=\left|\int_x^y f_k'(t)\,dt\right|\leq \sqrt{y-x}\|f_k\|_{L^2(\Omega)}\leq M\sqrt{y-x}.
\label{eq2.12.2}
\end{equation}

This shows that $(f_k)$ is equicontinuous. 

We prove that $f_k$ is uniformly bounded. First, we show that $(f_k(\alpha_i))$ is bounded. If not, passing to a subsequence we can assume $f_k(\alpha_i)$ converges to $\infty$. But then, using \eqref{eq2.12.2} we see that $f_k(x)$ converges to $\infty$ uniformly on $J_i$ and this contradicts the fact that $\|f_k\|_{L^2(\Omega)}$ is bounded. 

Thus $(f_k(\alpha_i))$ is bounded and, by \eqref{eq2.12.2}, $(f_k(x))$ is uniformly bounded. 

Then, by the Arzela-Ascoli theorem $f_k$ has a uniformly convergent subsequence, hence $f_n=Kz_n$ has a subsequence which converges in $L^2(\Omega)$. Thus $K$ is a compact operator. 
Since $K=\int \frac{1}{t-i}\,dP(t)$ it follows that $P$ is supported on a discrete subset $\Lambda$ and $P(\lambda)$ is finite dimensional for all $\lambda\in\Lambda$. 

Now take $\lambda\in\Lambda$ so $P(\lambda)\neq 0$. Take $f_\lambda$ in the eigenspace $P(\lambda)L^2(\Omega)$. Then $f_\lambda\in \domain(A)=\D_B$ and $Af_\lambda=\lambda f_\lambda$ so $\frac{1}{2\pi i}f_\lambda'=\lambda f_\lambda$. Solving the differential equation on each interval $J_i$ we obtain that 
$$f_\lambda(t)=\sum_{i=1}^nc_i\chi_{J_i}(t)e^{2\pi i\lambda t},\quad(t\in\Omega)$$
for some constants $c_1,\dots,c_n\in\bc$, not all of them zero. 

Since $f_\lambda$ is in the domain of $A$ which is $\D_B$, we must have $Bf_\lambda(\alpha)=f_\lambda(\beta)$. But $f_\lambda(\alpha_i)=c_ie^{2\pi i\lambda\alpha}$ and $f_\lambda(\beta_i)=c_i e^{2\pi i\lambda\beta}$ and therefore 
$BE(2\pi i\lambda\alpha)c= E(2\pi i\lambda\beta)c$. Then \eqref{eq2.12.1} follows, so $\lambda\in\Lambda_B$ and also the description of $P(\lambda)$.

Conversely, if $\lambda\in\Lambda_B$ then there exists a non-zero vector $c=(c_1,\dots,c_n)^*$ such that $BE(2\pi i\lambda\alpha)c=E(2\pi i\lambda\beta)c$. Define the function 
$$f_\lambda(t)=\sum_{i=1}^nc_i\chi_{J_i}(t)e^{2\pi i\lambda t},\quad(t\in\Omega).$$
Then $f_\lambda$ is in $\D_B$ which is the domain of $A$, and $Af_\lambda=\lambda f_\lambda$. Therefore $P(\lambda)\neq 0$ so $\lambda\in\Lambda$. 
\end{proof}

\section{Spectral sets}

Let $\Omega$ be a bounded open subset in $\br^k$ of finite positive Lebesgue measure. Consider the commuting minimal operators in $L^2(\Omega)$, i.e., the differential operators defined on the dense domain  $C_0^\infty$ in $L^2(\Omega)$. For the extension problem in this case \cite{Fug74}, one asks for existence of commuting self-adjoint extensions, a self-adjoint extension for each of the minimal Hermitian minimal operators. Now such systems of commuting self-adjoint extensions do not always exist (see e.g., \cite{Fug74}), but when they do, they generate a strongly continuous unitary representation of the additive group $\br^k$, acting on $L^2(\Omega)$.  And vice versa,  every such unitary representation $U$ comes from a system of commuting self-adjoint extensions. The joint spectrum of these extensions is called the spectrum of $U$.

    We focus on dimension $k=1$. We show (Theorem \ref{th3.1}) that a strongly continuous unitary representation $U$ of the additive group $\br$ exists, acting by local translations on $L^2(\Omega)$ (see Definition \ref{def3.3}), if and only if  $(\Omega,\Lambda)$ is a spectral pair where $\Lambda$ is the spectrum of the unitary representation $U$.  

    In Corollary \ref{cor3.6}  we characterize the special boundary matrices  $B$ (see section 2) for which $\Omega$ is spectral, i.e., there is a set $\Lambda$ such that $(\Omega, \Lambda)$ is a spectral pair.  We say that such a boundary matrix $B$ is spectral.

    Theorems \ref{th3.9} and \ref{th3.14} deal with the case of spectral pairs $(\Omega, \Lambda)$ when the spectrum $\Lambda$ has a finite period. In Theorem \ref{th3.13} we write out the associated boundary matrices.

\begin{theorem}\label{th3.1}
The set $\Omega=\cup_{i=1}^n(\alpha_i,\beta_i)$ is spectral if and only if there exists a strongly continuous one parameter unitary group $(U(t))_{t\in\br}$ on $L^2(\Omega)$ with the property that, for all $t\in\br$ and $f\in L^2(\Omega)$:
\begin{equation}
(U(t)f)(x)=f(x+t),\mbox{ for almost every }x\in \Omega\cap(\Omega-t).
\label{eq3.1.1}
\end{equation}
\end{theorem}
\begin{proof}
The necessity follows from Theorem \ref{th4.1}. For the converse, suppose $U(t)$ exists. Then, by Stone's Theorem \cite[Theorem 5.6, page 330]{Con90}, there exists a self-adjoint operator $A$ such that $U(t)=e^{2\pi itA}$. 
We claim that $A$ is a self-adjoint extension of $\restr{\Ds}{C_c^\infty(\Omega)}$.

\begin{lemma}\label{lem3.2}
If $f\in L^2(\Omega)$ is supported on $\Omega_\epsilon=\cup_{i=1}^n[\alpha_i+\epsilon,\beta_i-\epsilon]$ and $|t|<\epsilon$ then 
\begin{equation}
(U(t)f)(x)=f(x+t)\mbox{ for a.e. $x\in\Omega$}
\label{eq3.2.1}
\end{equation}
where $f(x):=0$ for $x$ not in $\Omega$.
\end{lemma}

\begin{proof}
If $|t|<\epsilon$ then $\Omega_\epsilon-t\subset (\Omega\cap(\Omega-t))$. Therefore $(U(t)f)(x)=f(x+t)$ for a.e. $x\in\Omega_\epsilon-t$. Put $(U(t)f)(x)=:g(x)$ for $x\in\Omega\setminus(\Omega_\epsilon-t)$. Then 
$$\|f\|_{L^2(\Omega)}^2=\|U(t)f\|_{L^2(\Omega)}^2=\int_{\Omega_\epsilon-t}|f(x+t)|^2\,dx+\int_{\Omega\setminus(\Omega_\epsilon-t)}|g(x)|^2\,dx$$$$=\int_{\Omega_\epsilon}|f(x)|^2\,dx+\int_{\Omega\setminus(\Omega_\epsilon-t)}|g(x)|^2\,dx=\|f\|_{L^2(\Omega)}^2+\|g\|_{L^2(\Omega)}^2$$
This implies that $g=0$ and \eqref{eq3.2.1} follows.
\end{proof}

Assume $f\in C_c^\infty(\Omega)$ is supported on $\Omega_\epsilon$ for some $\epsilon>0$. Using this lemma, for $|t|<\epsilon$ and for a.e. $x\in \Omega$, we have 
$$\frac{1}{2\pi i t}(U(t)f(x)-f(x))=\frac{1}{2\pi it}(f(x+t)-f(x))\mbox{ and this converges uniformly on $\Omega$ to $\frac{1}{2\pi i}f'(x)$}.$$
Therefore 
$$\frac{1}{2\pi i t}(U(t)f-f)\mbox{ converges in $L^2(\Omega)$ to $\Ds f$.}$$
This implies, see \cite[Theorem 5.1, page327]{Con90}, that $f$ is in the domain of $A$ and $A$ is an extension of $\Ds$.

Now, let $P$ be the spectral measure of $A$ as in Theorem \ref{th2.12}, $A=\int t\,dP(t)$ and let $B$ the unitary matrix associated to $A$ as in Theorem \ref{th1.7}. 
For $\lambda$ in the spectrum $\Lambda_B$ of $A$, we will prove that the space $P(\lambda)L^2(\Omega)$ is one-dimensional. 

Take an eigenfunction $f_\lambda$ which, by Theorem \ref{th2.12}, has to be of the form 
$$f_\lambda(t)=\sum_{i=1}^nc_i\chi_{J_i}(t)e^{2\pi i\lambda t},\quad(t\in\Omega).$$

Since $U(t)=e^{2\pi i tA}=\int e^{2\pi it x}\,dP(x)$ we have $U(t)f_\lambda=e^{2\pi i t\lambda}f_\lambda$. On the other hand $(U(t)f_\lambda)(x)=f_\lambda(x+t)$ for a.e. $x$ in $\Omega\cap (\Omega-t)$. Then we have for such points $x$:
$$\sum_{i=1}^nc_i\chi_{J_i}(x+t)e^{2\pi i\lambda(x+t)}=e^{2\pi i\lambda t}\sum_{i=1}^nc_i \chi_{J_i}(x)e^{2\pi i \lambda x}.$$
Now fix $j\neq k$ , $j,k\in\{1,\dots,n\}$ and pick $t$ such that $J_i\cap (J_k-t)$ has positive Lebesgue measure. Then pick an $x\in J_i\cap(J_k-t)$ such that the previous equation holds. We obtain $c_k=c_j$. 
Thus all the constants $c_i$ are the same $c_i=c$ and so $f_\lambda(t)=c e^{2\pi i \lambda t}$. Thus the eigenspace $P(\lambda)L^2(\Omega)$ is one-dimensional and is spanned by the function $e_\lambda$. 

Since the eigenspaces $P(\lambda)L^2(\Omega)$, $\lambda\in\Lambda_B$ span the entire space $L^2(\Omega)$ it follows that $\{e_\lambda :\lambda\in\Lambda_B\}$ is an orthogonal basis for $L^2(\Omega)$ and therefore $\Omega$ is a spectral set.

\end{proof}

\begin{definition}\label{def3.4}
Let $\mathbf 1$ be the constant vector $\mathbf 1=(1,1,\dots,1)$ in $\bc^n$. 
A unitary $n\times n$ matrix $B$ is called {\it spectral} for $\Omega$ if for every $\lambda\in\bc$ one has 
\begin{equation}
\operatorname*{ker}(I-E(-2\pi i\lambda\beta)BE(2\pi i\lambda\alpha))=\{0\}\mbox{ or }\bc\mathbf 1. 
\label{eq3.4.1}
\end{equation}

\end{definition}

\begin{theorem}\label{th3.4}
There is a one-to-one correspondence between the following objects:
\begin{enumerate}
	\item Spectra $\Lambda$ for the set $\Omega$.
	\item Unitary groups of local translations $U(t)$ for $\Omega$. 
	\item Self-adjoint extensions $A$ of the differential operator $\restr{\Ds}{C_c^\infty}(\Omega)$ with the property that all the eigenvectors of $A$ are of the form $ce_\lambda$, $c\in\bc$, $\lambda\in\sigma(A)$.
	\item Spectral unitary matrices $B$ for $\Omega$.  
\end{enumerate}
The correspondence from (i) to (ii) is given by $U(t)=U_\Lambda(t)$. The correspondence between (ii) and (iii) is given by $U(t)=e^{2\pi i tA}$. The correspondence from (iii) to (i) is given by $\Lambda=\sigma(A)$. 
The correspondence from (iii) to (iv) is given by Theorem \ref{th1.7}.
\end{theorem}

\begin{proof}
For the correspondence (i)$\leftrightarrow$(ii), we saw in the proof of Theorem \ref{th3.1} that the maps are well defined. We just have to show that the maps are inverse to each other. 

Let $\Lambda$ be a spectrum for $\Omega$, let $U_\Lambda(t)$ be the correspponding group of local translations and let $A$ be its infinitesimal generator given by Stone's theorem. We have to check that $\sigma(A)=\Lambda$. But we have $U_\Lambda(t)e_\lambda=e^{2\pi i\lambda t}e_\lambda$ so 
$$\frac{1}{2\pi i t}(U_\Lambda(t)e_\lambda-e_\lambda)=\frac{e^{2\pi i\lambda t}-1}{2\pi it}e_\lambda\rightarrow \lambda e_\lambda$$
uniformly and in $L^2(\Omega)$. Therefore $e_\lambda$ is in the domain of $A$ and $Ae_\lambda=\lambda e_\lambda$. Since $\{e_\lambda : \lambda\in\Lambda\}$ is an orthogonal basis, $A$ is diagonal in this basis and therefore $\sigma(A)=\Lambda$. 

Now take a group of local translations $U(t)$, let $A$ be its infinitesimal generator and let $\Lambda:=\sigma(A)$. From the proof of Theorem \ref{th3.1} we see that $\Lambda$ is a spectrum for $\Omega$ and we have to show that the group of local translations $U_\Lambda$ is $U$. But we have $Ae_\lambda=\lambda e_\lambda$ so $U(t)e_\lambda=e^{2\pi it} e_\lambda$. Also $U_\Lambda(t)e_\lambda=e^{2\pi i\lambda t}e_\lambda$. Since $\{e_\lambda : \lambda\in\Lambda\}$ is an orthogonal basis, we obtain that $U_\Lambda(t)=U(t)$.

The correspondence (ii)$\leftrightarrow$(iii) is given by Stone's theorem and the proof of Theorem \ref{th3.1}: If 
$U(t)$ is a group of local translations then we saw that the eigenfunctions of $A$ which are described in Theorem \ref{th2.12} must have all the coefficients $c_i$ equal, so they are of the form $ce_\lambda$ with $c\in\bc$. 

Conversely, if all the eigenfunctions are of this form, then $\{e_\lambda :\lambda\in\sigma(A)\}$ is an orthogonal basis for $L^2(\Omega)$, by Theorem \ref{th2.12}, and we have $U(t)e_\lambda=e^{2\pi i\lambda t}e_\lambda$ for $\lambda\in\sigma(A)$ and $t\in\br$. Then, looking again at the proof of Theorem \ref{th3.1}, we see that $U(t)$ is a group of local translations.

Finally, we check that the correspondence between self-adjoint extensions of $\Ds$ and unitary matrices $B$ given in Theorem \ref{th1.7} maps (iii) to (iv) and vice versa. 

Let $A$ be a self-adjoint extension as in (iii) and let $B$ be the unitary matrix associated to $A$ as in Theorem \ref{th1.7}. Let $\lambda\in\bc$. If $\det(I-E(-2\pi i\lambda\beta)BE(2\pi i\lambda\alpha))\neq 0$ then 
$\operatorname*{ker}(I-E(-2\pi i\lambda\beta)BE(2\pi i\lambda\alpha))=\{0\}$. If $\det(I-E(-2\pi i\lambda\beta)BE(2\pi i\lambda\alpha))\neq 0$ then take a vector $c$ in this kernel and the function 
$$\psi_\lambda(t)=\sum_{i=1}^nc_i\chi_{J_i}(t)e_\lambda(t)$$
is in the eigenspace $P(\lambda)$, by Theorem \ref{th2.12}. Then, by hypothesis, all the components $c_i$ have to be equal. So $\operatorname*{ker}(I-E(-2\pi i\lambda\beta)BE(2\pi i\lambda\alpha))=\bc\mathbf 1$.
So $B$ is spectral for $\Omega$.

Conversely, let $B$ be a spectral unitary matrix for $\Omega$ and let $A$ be the self-adjoint extension of $\Ds$ associated to $B$ as in Theorem \ref{th1.7}. By Theorem \ref{th2.12} all eigenfunctions are of the form 
$$\psi_\lambda(t)=\sum_{i=1}^nc_i\chi_{J_i}(t)e_\lambda(t),$$
with $c$ in $\operatorname*{ker}(I-E(-2\pi i\lambda\beta)BE(2\pi i\lambda\alpha))$. Since $B$ is spectral $c=\gamma\mathbf 1$ for some $\gamma\in\bc$, so $\psi_\lambda=\gamma e_\lambda$. So $A$ has the properties in (iii).

\end{proof}

\begin{corollary}\label{cor3.6}
For $\lambda\in\br$ define the diagonal $n\times n$ matrices $D_\alpha(\lambda)$ with diagonal entries $e^{2\pi i\lambda\alpha_i}$, $i=1,\dots, n$ and $D_\beta(\lambda)$ with diagonal entries $e^{2\pi i\lambda\beta_i}$, $i=1,\dots,n$. The set $\Omega$ is spectral if and only if there exists a spectral unitary matrix $B$ for $\Omega$. Moreover, any spectrum for $\Omega$ is given by 

\begin{equation}
\Lambda_B:=\left\{\lambda\in\bc : \det(I-D_\beta(\lambda)^*BD_\alpha(\lambda))=0\right\}=\left\{\lambda\in\br : Be_\lambda(\alpha)=e_\lambda(\beta)\right\},
\label{3.6.1}
\end{equation}
for a unique spectral unitary matrix $B$ for $\Omega$. 
\end{corollary}

\begin{proposition}\label{pr3.7}
Let $\Lambda$ be a spectrum for $\Omega$ and let $B$ be the corresponding spectral unitary matrix. Let $t_0\in\br$. Then the spectral unitary matrix for $\Lambda-t_0$ is $D_\beta(t_0)^*BD_\alpha(t_0)$.
\end{proposition}

\begin{proof}
We have, for all $\lambda\in\br$, 
$$I-D_\beta(\lambda)^*D_\beta(t_0)^*BD_\alpha(t_0)D_\alpha(\lambda)=I-D_\beta(\lambda+t_0)^*BD_\alpha(\lambda+t_0).$$
Therefore 
$$\Lambda_{D_\beta(t_0)^*BD_\alpha(t_0)}=\Lambda_B-t_0.$$
Since the correspondence between $\Lambda$ and $B$ is bijective, the result follows. 

\end{proof}

\begin{proposition}\label{pr3.8b}
Let $\Lambda$ be a spectrum for $\Omega$ and let $B$ be the corresponding spectral unitary matrix. Then $p$ is a period for $\Lambda$ if and only if 
\begin{equation}
BD_\alpha(p)=D_\beta(p)B.
\label{eq3.8b.1}
\end{equation}
\end{proposition}

\begin{proof}
The number $p$ is a period for $\Lambda$ iff $\Lambda-p=\Lambda$. The rest follows from Proposition \ref{pr3.7}. 
\end{proof}

\begin{definition}\label{def3.9}
For $p\in\br\setminus\{0\}$. Define the equivalence relation $\equiv_p$ on $\br$ by $t\equiv_p s$ iff $p(t-s)\in\bz$. 
\end{definition}

\begin{theorem}\label{th3.9}
Assume $\Omega$ is spectral with spectrum $\Lambda$ of period $p$. Assume in addition that $0\in\Lambda$. Let $B$ be the spectral unitary matrix associated to $\Lambda$. Then there exists a permutation $\sigma$ of the set $\{1,\dots, n\}$ with the property that $\beta_i\equiv_p\alpha_{\sigma(i)}$ for all $i=1,\dots,n$. The matrix $B$ has the following properties
\begin{enumerate}
	\item $b_{ij}=0$ if $\alpha_j\not\equiv_p\beta_i$.
	\item $\sum_{j: \alpha_j\equiv_p\beta_i}b_{ij}=1$ for all $i=1,\dots,n$.
	\item $\sum_{i: \beta_i\equiv_p\alpha_j}b_{ij}=1$ for all $j=1,\dots,n$.
	\item For every $\lambda\in\Lambda$, $\sum_{j: \alpha_j\equiv_p\beta_i}e^{-2\pi i(\beta_i-\alpha_j)\lambda}b_{ij}=1$ for all $i=1,\dots,n$.
	\item For every $\lambda\in\Lambda$, $\sum_{i: \beta_i\equiv_p\alpha_j}e^{-2\pi i(\beta_i-\alpha_j)\lambda}b_{ij}=1$ for all $j=1,\dots,n$.
	\item There is no proper subset $I$ of $\{1,\dots,n\}$, $I\neq \ty$, $I\neq \{1,\dots,n\}$ with the property that for all $i\in I, j\not\in I$, $\beta_i\not\equiv_p\alpha_j$.
\end{enumerate}

\end{theorem}

\begin{proof}
Since we have $\Lambda=\Lambda+p$ it follows from Proposition \ref{pr3.8b} that $D_\beta(p)^*BD_\alpha(p)=B$. This can be rewritten as $BD_\alpha(p)B^*=D_\beta(p)$. But, since $B$ is unitary, this means that the eigenvalues of $D_\alpha(p)$ and $D_\beta(p)$ are the same, with the same multiplicity so there is a permutation $\sigma$ of $\{1,\dots, n\}$ such that 
$e^{2\pi ip\beta_i}=e^{2\pi i p\alpha_{\sigma(i)}}$ which means that $\beta_i\equiv_p\alpha_{\sigma(i)}$.

Since we have $BD_\alpha(p)=D_\beta(p)B$, this means that $b_{ij}e^{2\pi i p\alpha_j}=e^{2\pi ip\beta_i}b_{ij}$ for all $i,j=1,\dots,n$. Therefore, if $\alpha_j\not\equiv_p\beta_i$ then $b_{ij}=0$. This proves (i).

(iv) follows from the fact that $B$ has to be spectral, so for $\lambda\in\Lambda$, the constant vector $\mathbf 1$ is an eigenvector with eigenvalue 1 for $D_\beta(\lambda)^*BD_\alpha(\lambda)$. 
Then the same has to be true for the adjoint of this unitary matrix and this implies (v). (ii) and (iii) are particular cases of (iv) and (v) for $\lambda=0$.

For (vi) assume there exists such a proper set $I$. The matrix $(b_{ij})_{i,j\in I}$ is unitary. Indeed, we have for $i,i'\in I$, using (i):
$$\delta_{ii'}=\sum_{j=1}^nb_{ij}\cj b_{i'j}=\sum_{j\in I}b_{ij}\cj b_{i'j}.$$

Then, since the matrix $B$ is unitary we have, for $j\in I$:
$$\sum_{i=1}^n|b_{ij}|^2=1=\sum_{i\in I}|b_{ij}|^2.$$
This implies that $b_{ij}=0$ for $i\not\in I$, $j\in I$.

Consider now the vector $a_i=0$ for $i\not\in I$, $a_i=1$ for $i\in I$. Then 
$$\sum_{j=1}^nb_{ij}a_j=\sum_{j\in I}b_{ij}=\left\{\begin{array}{cc}
0&\mbox{ if }i\not\in I\\
\sum_{j=1}^nb_{ij}=1&\mbox{ if }i\in I.
\end{array}\right.=a_i$$

Thus the vector $(a_1,\dots,a_n)$ is an eigenvector of eigenvalue $1$ for $B$. But this contradicts the fact that $B$ is spectral so $\mathbf 1$ is the only such eigenvector.
\end{proof}

\begin{corollary}\label{cor3.10}
Suppose $\Omega=\cup_{i=1}^n(\alpha_i,\beta_i)$ has a spectrum with period $p$. Then there exist some numbers $k_i\in\frac1p\bz$, $i=1,\dots,n$ such that $\Omega':=\cup_{i=1}^n[\alpha_i-k_i,\beta_i-k_i)$ is a disjoint union and  $\Omega'$ is a disjoint union of intervals with lengths in $\frac1p\bz$. 
\end{corollary}

\begin{proof}
By translating, we can assume the spectrum contains 0. Then, by Theorem \ref{th3.9}, there exists a permutation $\sigma$ such that $\beta_i\equiv_p\alpha_{\sigma(i)}$, $i=1,\dots,n$. 
Take the interval $[\alpha_1,\beta_1)$ and define $k_1:=0$. We have $\beta_1\equiv\alpha_{\sigma(1)}$. If $\sigma(1)=1$ we stop. If not, we have $k_{\sigma(1)}:=\alpha_{\sigma(1)}-\beta_1\in\frac1p\bz$ and note that the intervals  $[\alpha_1,\beta_1)$ and $[\alpha_{\sigma(1)}-k_{\sigma(1)},\beta_{\sigma(1)}-k_{\sigma(1)})$ have a common endpoint. By induction suppose we defined $k_{\sigma^i(1)}\in\frac1p\bz$ for $i=0,\dots,l$ such that the intervals 
$[\alpha_{\sigma^i(1)},\beta_{\sigma^i(1)})-k_{\sigma^i(1)}$ are disjoint and two consecutive ones have a common endpoint and $\sigma^i(1)\neq 1$. If $\sigma^{l+1}(1)=\sigma^i(1)$ for some $i\leq l$ then $\sigma^{l+1-i}(1)=1$ and, by the induction hypothesis, this is possible only when $i=0$. In this case we stop and we closed a cycle for $1$. If this is not the case then we have $k_{\sigma^{l+1}(1)}:=\alpha_{\sigma^{l+1}(1)}-\beta_{\sigma^l(1)}+k_{\sigma^l(1)}\in\frac{1}{p}\bz$. Therefore the intervals $[\alpha_{\sigma^l(1)},\beta_{\sigma^l(1)})-k_{\sigma^l(1)}$ and
$[\alpha_{\sigma^{l+1}(1)},\beta_{\sigma^{l+1}(1)})-k_{\sigma^{l+1}(1)}$ have a common endpoint.  

When we close a cycle, say after $l$ steps, we compute the length of the interval formed with the union of $[\alpha_{\sigma^i(1)},\beta_{\sigma^i(1)})-k_{\sigma^i(1)}$, $i=1,\dots,l$. This is 
congruent in $\frac1p\bz$ to 
$$\beta_1-\alpha_1+\beta_{\sigma(1)}-\alpha_{\sigma(1)}+\dots+\beta_{\sigma^l(1)}-\alpha_{\sigma^l(1)}\equiv_p\beta_1-\alpha_{\sigma^l(1)}=\beta_{\sigma^{l+1}(1)}-\alpha_{\sigma^l(1)}\equiv_p 0.$$
Hence the length of this interval is in $\frac1p\bz$. 

If this cycle $1,\sigma(1),\dots, \sigma^l(1)$ covers the entire set $\{1,\dots,n\}$ we are done. If not, pick a point outside the cycle and repeat the same procedure; the first interval can be translated far enough by an element in $\frac1p\bz$ to ensure the resulting interval is disjoint from the one constructed for the previous cycles.

\end{proof}

\begin{proposition}\label{pr3.12}
Let $\Lambda$ be a spectrum for $\Omega$ and let $B$ be the associated spectral matrix. For $t:=(t_1,\dots,t_n)\in\br^n$, define the matrices 
\begin{equation}
M_{\alpha,t}:=\frac{1}{\sqrt{N}}\left(e^{2\pi i\alpha_i t_j}\right)_{i,j=1}^n,\quad
M_{\beta,t}:=\frac{1}{\sqrt{N}}\left(e^{2\pi i\beta_i t_j}\right)_{i,j=1}^n.
\label{eq3.12.1}
\end{equation}
Then for any vector $\lambda=(\lambda_1,\dots,\lambda_n)$ in $\Lambda^n$, the matrix $M_{\alpha,\lambda}$ is non-singular/unitary if and only of $M_{\beta,\lambda}$ is non-singular/unitary and in this case 
\begin{equation}
B=M_{\beta,\lambda}M_{\alpha,\lambda}^{-1}.
\label{eq3.12.2}
\end{equation}

\end{proposition}

\begin{proof}
Using Theorem \ref{th3.9} for $\lambda_1,\dots,\lambda_n$, we obtain for $i,k=1,\dots,k$:
$$\sum_{j=1}^nb_{ij}e^{2\pi i\alpha_j\lambda_k}=e^{2\pi i\beta_i\lambda_k}.$$
This means that $BM_{\alpha,\lambda}=M_{\beta,\lambda}$. Since $B$ is unitary, the equivalence follows and also \eqref{eq3.12.2}
\end{proof}

\begin{theorem}\label{th3.13}
We use the same notation as in Proposition \ref{pr3.12}. The set $\Omega$ is spectral if and only if there exist points in $\br$, $\lambda=(\lambda_1,\dots,\lambda_n)$ such that $M_{\alpha,\lambda}$ is non-singular and the matrix 
\begin{equation}
B=M_{\beta,\lambda}M_{\alpha,\lambda}^{-1}
\label{eq3.13.1}
\end{equation}
is a spectral unitary matrix for $\Omega$. 
\end{theorem}

\begin{proof}
The reverse implication follows from Corollary \ref{cor3.6}. For the direct implication, let $\Lambda$ be a spectrum for $\Omega$. Let $A$ be the associated self-adjoint extension of $\Ds$ and let $B$ be the corresponding spectral unitary matrix, see Theorem \ref{th3.4}. By Proposition \ref{pr3.12} it is enough to show that there exists $\lambda=(\lambda_1,\dots,\lambda_n)$ with $\lambda_i\in\Lambda$ such that the matrix $M_{\alpha,\lambda}$ is non-singular. 

Take an arbitrary $f$ in the domain of $A$. Since $\{e_\lambda :\lambda\in\Lambda\}$ is an orthogonal basis, there exist $c_\lambda\in\bc$ such that $\sum_\lambda c_\lambda e_\lambda=f$ with convergence in $L^2(\Omega)$. Then, using the spectral decomposition of $A$ (Theorem \ref{th2.12}), we have that $f=\sum_{\lambda\neq 0}\frac{c_\lambda}{\lambda}e_\lambda+c_0e_0$, where $c_0=0$ if $0$ is not in $\Lambda$ and $c_0=\frac{1}{|\Omega|}\ip{f}{e_0}$ if $0\in\Lambda$. 
This shows that $\sum_{\lambda\neq 0} \frac{c_\lambda}{\lambda} e_\lambda +c_0e_0$ converges to $f$ in the graph inner product. Then with Lemma \ref{lem1.9} we see that $\sum_{\lambda\neq 0}c_\lambda e_\lambda(\alpha)+c_0e_0(\alpha)$ converges in $\bc^n$ to $f(\alpha)$. Since $f$ was arbitrary in the domain of $A$, we get that the vectors $(e_\lambda(\alpha))_{\lambda\in\Lambda}$ span $\bc^n$. So we can pick $\lambda_1,\dots,\lambda_n$ such that the vectors $e_{\lambda_1}(\alpha),\dots,e_{\lambda_n}(\alpha)$ are linearly independent. This implies that the matrix $M_{\alpha,\lambda}$ is non-singular. 

\end{proof}

\begin{theorem}\label{th3.14}
Assume that $\Omega$ is spectral with spectrum $\Lambda$ that contains $0$ and has period $p$. Suppose there exists an $i\in\{1,\dots,n\}$ such that there exists a unique $j\in\{1,\dots,n\}$ with $\beta_i\equiv_p\alpha_j$. Then 
\begin{enumerate}
	\item The spectrum $\Lambda$ is contained in $\frac{1}{\beta_i-\alpha_j}\bz$. 
	\item The translate of the interval $J_j$, $J_j'=J_j+\beta_i-\alpha_j=(\beta_i,\beta_j+\beta_i-\alpha_j)$, which has a common endpoint with the interval $J_i$, is disjoint from all $J_k$ with $k\neq i,j$. 
	\item The set $\Omega':=J_j'\cup_{k\neq j}J_k$ is spectral with the same spectrum $\Lambda$. 
	\item The statements analogous to (ii) and (iii) also hold when one replaces the interval $J_i$ by $J_i'=J_i+\alpha_j-\beta_i=(\alpha_i+\alpha_j-\beta_i,\alpha_j)$.

\end{enumerate}

\end{theorem}

\begin{proof}
By Theorem \ref{th3.9}(i) we have that $b_{ik}=0$ for all $k\neq j$ and by Theorem \ref{th3.9}(ii) we get that $b_{ij}=1$. Then, by Theorem \ref{th3.9}(iv), we get that for any $\lambda\in\Lambda$ we must have $e^{-2\pi i(\beta_i-\alpha_j)\lambda}=1$ so $(\beta_i-\alpha_j)\lambda\in\bz$. This proves (i).
Also, it shows that, for $\lambda\in\Lambda$, $e_{\lambda}(t\pm(\beta_i-\alpha_j))=e_{\lambda}(t)$, $t\in\br$. 

Consider now the function $h:=\frac{1}{|\Omega|}\left(\chi_{J_j'}+\sum_{k\neq j}\chi_{J_k}\right)$. We will show that the measure $h(t)\,dt$ is a spectral measure with spectrum $\Lambda$. Note first that $\mu$ is a probability measure. For $f\in L^2(\mu)$, $\lambda\in\Lambda$, define the function $\Psi f$ on $\Omega$, $\Psi(f)(t)=f(t)$ if $t\in\cup_{k\neq j}J_k$ and $\Psi(f)(t)=f(t+\beta_i-\alpha_j)$ if $t\in J_j$. 
Note that $\Psi(e_\lambda)=\restr{e_\lambda}{\Omega}$. Also, $\Psi$ is an isometry from $L^2(\mu)$ onto $L^2(\Omega)$, up to the normalization constant $\frac{1}{|\Omega|}$. Therefore $\{e_\lambda : \lambda\in\Lambda\}$ is a spectrum for $L^2(\mu)$. But then this means that $h$ is constant on its support (see \cite{DuLa12,Lai11}). So $J_j'$ is disjoint from $J_k$ for $k\neq j$ and $\Omega'$ is spectral. 

(iv) follows in the same way.

\end{proof}

\begin{theorem}\label{th4.10}
Let $p\in\bn$ and let $\Omega=\cup_{i=0}^{p-1}(\frac{\alpha_i}p,\frac{\alpha_i+1}p)$, with $|\Omega|=1$, $\alpha_i\in\bz$ for all $i=0,\dots,p-1$, $0=\alpha_1<\alpha_2<\dots<\alpha_p$. Assume that $\Omega$ has spectrum $\Lambda$ with period $p$ (see Proposition \ref{pr3.9}), $\Lambda=\{\lambda_0=0,\lambda_1,\dots,\lambda_{p-1}\}+p\bz$. Let $U$ be the associated group of local translations and $B$ the spectral unitary matrix associated to $\Lambda$ as in Theorem \ref{th1.7}. Then  
\begin{equation}
B=M_{\alpha}D(\lambda)M_{\alpha}^*\mbox{ where }M_\alpha=\frac{1}{\sqrt p}\left(e^{2\pi i\frac1p\alpha_i\lambda_j}\right)_{i,j=0}^{p-1}, \quad D(\lambda)=\begin{bmatrix}
e^{2\pi i\frac1p\lambda_0}& &0\\
 &\ddots& \\
 0& &e^{2\pi i\frac1p\lambda_{p-1}}
 \end{bmatrix}
\label{eq4.10.1c}
\end{equation}
For $x\in\br$, let $a(x)$ be the unique integer such that $[x]_p:=x-\frac {a(x)}p\in[0,\frac1p)$. Then for $f\in L^2(\Omega)$, $t\in\br$,
\begin{equation}
\begin{pmatrix}
	(U(t)f)(x+\frac{\alpha_0}p)\\
	\vdots\\
	(U(t)f)(x+\frac{\alpha_{p-1}}p)
\end{pmatrix}
=B^{a(x+t)}
\begin{pmatrix}
	f([x+t]_p+\frac{\alpha_0}p)\\
	\vdots\\
	f([x+t]_p+\frac{\alpha_{p-1}}p)
\end{pmatrix},\quad(x\in(0,\frac1p))
\label{eq4.10.2c}
\end{equation}
\end{theorem}

\begin{proof}
Equation \eqref{eq4.10.1c} follows from Proposition \ref{pr3.12}.

We use Theorem \ref{pr4.8}. For $x\in(0,\frac1p)$ the sets $\Omega_x$ are all equal to $\{\alpha_0,\dots,\alpha_{p-1}\}$. Also $k_i(x)=\alpha_i-a(x)$, $x\in\br$, $i=0,\dots,p-1$. Therefore 
$$\M_x=\frac{1}{\sqrt{p}}\left(e^{2\pi i \frac{\alpha_i-a(x)}p\cdot\lambda_j}\right)_{i,j=0}^{p-1}=M_\alpha D(\lambda)^{-a(x)}.$$

Then, for $x\in[0,\frac1p)$, $a(x)=0$ and 
$$\begin{pmatrix}
	(U(t)f)(x+\frac{\alpha_0}p)\\
	\vdots\\
	(U(t)f)(x+\frac{\alpha_{p-1}}p)
\end{pmatrix}
=(WU(t)f)(x)=(\U_p(t)Wf)(x)=\M_x\M_{x+t}^*(Wf)(x+t)$$$$=M_\alpha D(\lambda)^{a(x+t)}M_\alpha^*
\begin{pmatrix}
	f([x+t]_p+\frac{\alpha_0}p)\\
	\vdots\\
	f([x+t]_p+\frac{\alpha_{p-1}}p)
\end{pmatrix}$$
This is \eqref{eq4.10.2c}.
\end{proof}

\section{Two intervals}

\begin{theorem}\label{th6.1}
Let $\Omega$ be a union of two intervals, $|\Omega|=1$, $\Omega=(0,w)\cup(w+\rho,1+\rho)$, $0<w<1$, $\rho>0$. Then $\Omega$ is spectral if and only if one the following two conditions holds:
\begin{enumerate}
	\item $\rho$ is an integer and $w\neq\frac12$. In this case the only spectrum that contains $0$ is $\Lambda=\bz$, $\Omega$ is $\bz$-translation congruent to $(0,1)$ and it tiles $\br$ by $\bz$. The spectral unitary matrix $B$ associated to $\Lambda$ as in Theorem \ref{th3.4} is 
	$$B=\begin{pmatrix}
	0&1\\
	1&0
	\end{pmatrix}.$$
	
	For $x\in\br$ denote by $x\mod\Omega$ the unique point in $\Omega$ such that $x-x\mod\Omega\in\bz$. Then the group of local translations $U_\Lambda$ is 
	\begin{equation}
(U_\Lambda(t)f(x)=f((x+t)\mod\Omega),\quad(x\in\Omega, f\in L^2(\Omega)).
\label{eq6.1.0}
\end{equation}

	\item $w=\frac12$, $w+\rho=\frac l2\in\frac12\bz$, $1+\rho=\frac{l+1}2$. In this case, any spectrum that contains $0$ is of the form 
	\begin{equation}
\Lambda=\{0,\frac{2k+1}{l}\}+2\bz,
\label{eq6.1.1}
\end{equation}
for some $k\in\{ 0,\dots,l-1\}$. 
\end{enumerate} 
The associated spectral unitary matrix is 
$$B=\begin{pmatrix}
	\frac{1+\xi}{2}&\frac{1-\xi}{2}\\ \frac{1-\xi}2&\frac{1+\xi}2
\end{pmatrix},$$
where $\xi=e^{\pi i\frac{2k+1}{l}}.$

Let $a(x)$ be the unique integer such that $x\in[\frac{a(x)}2,\frac{a(x)+1}2)$, $x\in\br$ and let $[x]_2=x-\frac{a(x)}2\in[0,\frac12)$. The group of local translations is 
\begin{equation}
(U_\Lambda(t)f)(x)=\frac{1+\xi^{a(x+t)}}2f([x+t]_2)+\frac{1-\xi^{a(x+t)}}2f([x+t]_2+\frac l2)\quad\mbox{ if }x\in[0,\frac12)
\label{eq6.1.3}
\end{equation}
\begin{equation}
(U_\Lambda(t)f)(x+\frac l2)=\frac{1-\xi^{a(x+t)}}2f([x+t]_2)+\frac{1+\xi^{a(x+t)}}2f([x+t]_2+\frac l2)\quad\mbox{ if }x\in[0,\frac 12).
\label{eq6.1.4}
\end{equation}
In a shorter form 
\begin{equation}
\begin{pmatrix}
	(U(t)f)(x)\\
	(U(t)f)(x+\frac l2)
\end{pmatrix}= B^{a(x+t)}\begin{pmatrix}
	f([x+t]_2)\\
	f([x+t]_2+\frac l2)
\end{pmatrix},\quad(x\in[0,\frac12)).
\label{eq6.1.5}
\end{equation}
\end{theorem}

\begin{proof}
Let $\Lambda$ be a spectrum for $\Omega$, that contains $0$. Then $\Lambda$ has some period $p$ in $\bn$, $\Lambda=\{\lambda_0=0,\dots,\lambda_{p-1}\}+p\bz$, with $\lambda_0,\dots,\lambda_{p-1}$ distinct in $[0,p)$. 

Suppose first that $0\not\equiv_p w$. Then, by Theorem \ref{th3.9}, $0\equiv_p 1+\rho$ and the matrix $B$ has to be of the form $$B=\begin{pmatrix}
	0&1\\
	1&0
	\end{pmatrix}.$$
	
	By Corollary \ref{cor3.6}, we have 
	$$\Lambda=\left\{\lambda\in\br : \begin{pmatrix}
	0&1\\
	1&0
	\end{pmatrix}\begin{pmatrix}1\\ e^{2\pi i (w+\rho)\lambda}\end{pmatrix}=\begin{pmatrix} e^{2\pi iw\lambda}\\ e^{2\pi i(1+\rho)\lambda}\end{pmatrix}\right\}=\{\lambda : e^{2\pi i(w+\rho)\lambda}=e^{2\pi iw\lambda}, 1=e^{2\pi i(1+\rho)\lambda}\} $$
	$$=\{\lambda : \rho\lambda, (1+\rho)\lambda\in\bz\}=\{\lambda : \lambda,\rho\lambda\in\bz\}.$$
	
	But this implies that $\lambda_0,\dots,\lambda_{p-1}$ are in $\bz\cap[0,p)$ and, since they are all distinct, they must be $0,\dots,p-1$. So $\Lambda=\bz$. Since $1\in\Lambda$ we get that $\rho\in\bz$. 
	Then we see that a translation of $(w+\rho,1+\rho)$ by the integer $-\rho$ will transform $\Omega$ into $(0,1)$ so $\Omega$ is $\bz$-translation congruent to $(0,1)$ and also it tiles $\br$ by $\bz$.

	Now assume $0\equiv_p w$. By Theorem \ref{th3.9}(vi), we cannot have $0\not\equiv_p 1+\rho$. Therefore, in this case $w,w+\rho,1+\rho\in\frac1p\bz$. 
	
	Let 
	$$\mbox{$w=\frac kp$, $w+\rho=\frac ap$ and $\rho+1=\frac{a+r}p$ with $a,p,r\in\bn$, $0<k<p$, $a>k$, $r>0$, $k+r=p$.}$$ 
	
	Since $\Omega$ is a union of $p$ intervals of the form $(\frac ip,\frac{i+1}p)$, it $p$-tiles $\br$ by $\frac1p\bz$-translations. The sets $\Omega_x$, $x\in(0,\frac1p)$ defined in Theorem \ref{pr4.8}, are 
	$$\Omega_x=\{0,1,\dots,k-1\}\cup\{a,a+1,\dots,a+r-1\}=:S,$$
	and they have spectrum $\frac1p\{\lambda_0,\dots,\lambda_{p-1}\}$.
	
	By \cite[Theorem 1.5]{DuJo12a}, there exists a Hilbert space $\H$, a vector $v_0\in \H$ and a strongly continuous one-parameter group $(U(t))_{t\in\br}$ such that $\{U(t)v_0 : t\in S\}$ is an orthonormal basis for $\H$. 
	
	From \cite[Theorem 1.5, Remark 2.1]{DuJo12a}, we see that $U(t)$ is the multiplication operator by the function $e_t$ on the discrete $L^2$-space of the spectrum $\frac1p\{\lambda_0,\dots,\lambda_{p-1}\}$. So the spectrum of $U(t)$ is 
	$$\{e^{2\pi i t\frac1p\lambda_i} :i=0,\dots,p-1\}.$$

	We will use the following notation: $U(i)v_0=:v_i$, for $i\in S$; we denote by $\{i_1,\dots,i_q\}$ the linear span of the vectors $\{v_{i_1},\dots, v_{i_q}\}$. We write 
	$$\{i_1,\dots, i_q\}\nar{t} \{j_1,\dots,j_s\}\quad \mbox{ if $U(t)$ maps the first subspace into the second}.$$

	Since $\{0,\dots,k-2\}\cup \{a,\dots,a+r-2\}\nar{1}\{1,\dots, k-1\}\cup\{a+1,\dots,a+r-1\}$, taking orthocomplements we must have $\{k-1,a+r-1\}\nar{1}\{0,a\}$.
	
	We distinguish several cases:

	{\it Case 1.} $r<k$.

	We have $r\nar{a-1}a+r-1\nar{1}\{0,a\}$ so $r\nar{a}\{0,a\}$. But, $0\nar{a}a$ and therefore $r\nar{a}0$. 
	Then $a+r-1\nar{1-a}r\nar{a}0$ so $a+r-1\nar 1 0$. Finally, taking complements we must have $k-1\nar 1 a$.
	
	Therefore the unitary $U(1)$ permutes cyclically the subspaces $0,1,\dots,k-1,a,a+1,\dots,a+r-1$ moving each one to the next in line and the last one to $0$. Note also that for all except $k-1$ and $a+r-1$ we have $U(1)v_i=v_{i+1}$ and 
	$U(1)v_{k-1}=\alpha v_a$, $U(1)v_{a+r-1}=\beta v_0$ for some $|\alpha|=|\beta|=1$. 
	
	Then $U(p)=\alpha\beta I$. But since 1 is in the spectrum of $U(1)$, we get that $\alpha\beta=1$. From this we see also that $j\nar l j$ iff $l$ is divisible by $p$, for any $j$. But we have 
	$k-1\nar{a-k+1} a$ and $k-1\nar 1 a$ so $k-1\nar{ a-k+1} a\nar{-1}k-1$, therefore $a-k$ is divisible by $p$, say $a-k=pl$ for some $l\in\bz$. Note that this implies that $\rho=\frac{a-k}p=l$ is an integer and therefore we are in the case (i), $\Omega$ is $\bz$-translation congruent to $(0,1)$ and tiles $\br$ by $\bz$-translations.  
	
	Then we have 
	$$v_a=U(a-k+1)v_{k-1}=U(lp+k-k+1)v_{k-1}=U(1)v_{k-1}=\alpha v_a.$$
	This implies that $\alpha=1$ so $\beta=1$. 
	
	Since $U(p)=1$ it follows that its spectrum is 
	$$\{1\}=\{e^{2\pi i\lambda_i} : i=0,\dots,p-1\}.$$
	So $\lambda_i$ are all in $\bz\cap[0,p)$. Since they are distinct, we get that they cover $0,\dots,p-1$ and so $\Lambda=\bz$. 
	
	Let us compute the matrix $B$ in this case. We use Proposition \ref{pr3.12}. We take $\lambda_0=0$ and $\lambda_1=1$. We have that the matrix 
	$$\M_\alpha=\begin{pmatrix}
	1&1\\
	1&e^{2\pi i(w+\rho)}\end{pmatrix}=\begin{pmatrix}
	1&1\\
	1&e^{2\pi iw}\end{pmatrix}$$
	is non-singular $(0<w<1)$. Therefore 
	$$B=M_\beta M_\alpha^{-1}=\begin{pmatrix}
	1&e^{2\pi iw}\\
	1&e^{2\pi i(1+\rho)}\end{pmatrix}\frac{1}{e^{2\pi iw}-1}\begin{pmatrix}
	e^{2\pi iw}&-1\\
	-1&1\end{pmatrix}=\begin{pmatrix}
	0&1\\
	1&0\end{pmatrix}$$

	{\it Case 2}. $k<r$.

	This can be treated similarly to Case 1. The role of the two intervals is reversed and we get the results in (i).

	{\it Case 3.} $k=r$. 
	
	In tis case we get that $p=2k$ so $w=\frac12$. Then, as in the previous cases we get that $\{k-1,a+k-1\}\nar{1}\{0,a\}$. Therefore $U(1)$ permutes cyclically the subspaces $\{i,a+i\}$ $i=0,\dots,k-1$ and if , for some $i=0,\dots,k-1$, $U(j)v_i$ is in $\{i,a+i\}$ then $j$ is a multiple of $k$. But since $U(a)v_0=v_a$, this implies that $a$ is a multiple of $k$, $a=kl$ for some $l\in\bz$. Then $w=\frac12$, $w+\rho=\frac ap=\frac{l}{2}$.
	
	Since all endpoints are in $\frac12\bz$ it follows by Proposition \ref{pr3.9} that any spectrum has period 2. So we can take $p=2$. Then the sets $\Omega_x$ for $[0,\frac12)$ are all 
	$$\Omega_x=\{0,l\}.$$
	
	We have that $\frac12\{0,\lambda_1\}$ is a spectrum for $\Omega_x$ so 
	$1+e^{2\pi i\frac12\lambda_1\cdot l}=0$ which is equivalent to $\lambda_1\cdot l$ is an odd integer, $\lambda_1=\frac{2k+1}l$. 
	Therefore any spectrum for $\Omega$ is of the form \eqref{eq6.1.1}. Since $\Lambda$ has period 2, we can also replace $\frac{2k+1}l$ by $\frac{2k+1}l-2a$ to get the same set $\Lambda$ and therefore we can pick $0\leq k<l$.

	Conversely, any set of this form is a spectrum, by Theorem \ref{pr4.8}.
	
	To find the unitary spectral matrix $B$ we use Proposition \ref{pr3.12}. Take $\lambda_0=0$, $\lambda_1=\frac{2k+1}l$. 
	Then 
	$$M_\alpha=\begin{pmatrix}
	1&1\\
	1&e^{2\pi i\frac{2k+1}l\cdot \frac{l}{2}},\end{pmatrix}
	=\begin{pmatrix}
	1&1\\
	1&-1
\end{pmatrix},\quad
	M_\beta=
\begin{pmatrix}
	1& e^{2\pi i\frac{2k+1}l\cdot\frac12}\\
	1&e^{2\pi i \frac{2k+1}l\cdot\frac{l+1}2}
\end{pmatrix}
=\begin{pmatrix}
	1&\xi\\
	1&-\xi
\end{pmatrix}.
$$
Then 
$$B=M_{\beta}M_{\alpha}^{-1}=\begin{pmatrix}1&\xi\\1&-\xi\end{pmatrix}\cdot
\frac12 \begin{pmatrix}
	1&1\\1&-1
\end{pmatrix}=\begin{pmatrix}
	\frac{1+\xi}{2}&\frac{1-\xi}{2}\\ \frac{1-\xi}2&\frac{1+\xi}2
\end{pmatrix}.
$$

	To determine the group of local translations, we use Theorem \ref{pr4.8}. In case (i), the spectrum is $\bz$ and the period is $1$. The matrices $\mathcal M_x$ are $\mathcal M_x=1$. Therefore $(\U_1(t)F)(x)=F(x+t)$, $x,t\in\br$, $F\in L^2([0,1),\bc^2)$. The isomorphism $W:L^2(\Omega)\rightarrow L^2([0,1),\bc^2)$ is $(Wf)(x)=f(x\mod\Omega)$. Then \eqref{eq4.8.7} implies \eqref{eq6.1.0}.
	
	For case (ii), we have that the period $p=2$, $k_0(x)=-a(x)$, $k_1(x)=-a(x)+l$, $\lambda_0=0$, $\lambda_1=\frac{2k+1}{l}$.
	$$\M_x=\frac1{\sqrt2}\begin{pmatrix}
	1&e^{2\pi i \frac{2k+1}l\frac{-a(x)}2}\\
	1&e^{2\pi i \frac{2k+1}l\frac{-a(x)+l}2}
\end{pmatrix}
=\frac{1}{\sqrt{2}}\begin{pmatrix}
	1&\xi^{-a(x)}\\
	1&-\xi^{-a(x)}
\end{pmatrix}
$$
Then 
$$\M_x\M_{x+t}^*=\frac12 \begin{pmatrix}
	1+\xi^{a(x+t)-a(x)}&1-\xi^{a(x+t)-a(x)}\\
	1-\xi^{a(x+t)-a(x)}&1+\xi^{a(x+t)-a(x)}.
\end{pmatrix}$$
	For $x\in [0,\frac12)$, $a(x)=0$ and for $F\in L^2([0,\frac12),\bc^2)$,
	$$(\U_2(t)Wf)(x)=\M_x\M_{x+t}^*(Wf)(x+t)=\frac12 \begin{pmatrix}
	1+\xi^{a(x+t)}&1-\xi^{a(x+t)}\\
	1-\xi^{a(x+t)}&1+\xi^{a(x+t)}
\end{pmatrix}
\begin{pmatrix}
	f([x+t]_2)\\
	f([x+t]_2+\frac l2)
\end{pmatrix}
$$
But, by \eqref{eq4.8.7}, this has to equal 
$$(WU(t)f)(x)=\begin{pmatrix}
	(U(t)f)(x)\\
	(U(t)f)(x+\frac l2)
\end{pmatrix}
	$$
	This implies \eqref{eq6.1.3}, \eqref{eq6.1.4}.

	By induction, one can check that 
	$$B^k=\begin{pmatrix}
	\frac{1+\xi^k}2&\frac{1-\xi^k}2\\
	\frac{1-\xi^k}2&\frac{1+\xi^k}2
\end{pmatrix}$$
and this implies \eqref{eq6.1.5}
\end{proof}

\begin{acknowledgements}
This work was partially supported by a grant from the Simons Foundation (\#228539 to Dorin Dutkay).
This work was done while the first named author (PJ) was visiting the University of Central Florida. We are grateful to the UCF-Math Department for hospitality and support. The authors are pleased to thank Professors Deguang Han and Qiyu Sun for helpful conversations. PJ was supported in part by the National Science Foundation, via a Univ of Iowa VIGRE grant. 
\end{acknowledgements}
\bibliographystyle{alpha}
\bibliography{eframes}

\end{document}